\documentclass[12pt]{amsart} 

\usepackage{amsmath,amssymb}
\usepackage{amsthm,amscd}

\newtheorem{theorem}{Theorem}[section]
\newtheorem{cor}[theorem]{Corollary}
\newtheorem{lem}[theorem]{Lemma}
\newtheorem{pro}[theorem]{Proposition}
\newtheorem{rem}[theorem]{Remark }
\newtheorem{ex}[theorem]{Example }

\begin{document}

\title[Green functions and weights]{Green functions and weights
       of polynomial skew products on $\mathbb{C}^2$}
\author[K. Ueno]{Kohei Ueno}
\address{Toba National College of Maritime Technology, 
         Mie 517-8501, Japan}
\curraddr{}
\email{kueno@toba-cmt.ac.jp}
\urladdr{}
\dedicatory{}
\date{}
\thanks{}
\translator{}

\begin{abstract}
We study the dynamics of polynomial skew products on $\mathbb{C}^2$.
By using suitable weights, 
we prove the existence of several types of Green functions. 
Largely, 
continuity and plurisubharmonicity follow.
Moreover, 
it relates to the dynamics of the rational extensions 
to weighted projective spaces.
\end{abstract}

\maketitle

\section{Introduction}

The Green function $G_f$ of $f$ is a powerful tool
for the study of the dynamics of a polynomial map $f$ 
from $\mathbb{C}^{2}$ to $\mathbb{C}^{2}$,
which is defined as
\[
G_f(z,w) = \lim_{n \to \infty} \frac{1}{{\lambda}^n} \log^{+} |f^n(z,w)|,
\]
where $f^n$ denotes the $n$-th iterate of $f$,
$\lambda$ is the dynamical degree of $f$,
and $|(z,w)| = \max \{ |z|, |w| \}$.
In one dimension, it is well known 
that the Green function $G_p$ of a polynomial $p$ is defined,
continuous and subharmonic on $\mathbb{C}$ and 
that it coincides with the Green function of $K_p$ with pole at infinity,
where $K_p$ is the set of points whose orbits under $p$ are bounded.
However, in two dimension, it is not known 
whether the limit $G_f$ is defined on $\mathbb{C}^{2}$. 
One strategy to study the dynamics in two dimension is
to assume that $f$ has a good property that induces the existence of $G_f$.
For example, 
it is well known that if $f$ is regular, that is, 
if $f$ extends to a holomorphic map on the projective space $\mathbb{P}^{2}$,
then $G_f$ is defined,
continuous and plurisubharmonic on $\mathbb{C}^2$.
Moreover, it coincides with the pluricomplex Green function of $K_f$,
where $K_f$ is the set of points whose orbits under $f$ are bounded
(see e.g. \cite{bj}). 
To relax this regularity condition,
many authors have paid attention to 
another condition called algebraically stability
(e.g. \cite{fs}, \cite{g}, \cite{dg}, \cite{ddg1}, \cite{ddg2} and \cite{ddg3}).
Recently, 
Favre and Jonsson~\cite{fj} proved that any polynomial map has an extension 
that is algebraically stable in some weak sense, 
to a projective compactification of $\mathbb{C}^{2}$ 
with at worst quotient singularities. 
Their proof is based on the valuative techniques developed in \cite{fj-ev},
with which they proved that if $f$ is not conjugate to a skew product, 
then there is a plurisubharmonic function $u$,
which is close to $\log^{+} |(z,w)|$,
such that $\lambda^{-n} u \circ f^n$
decreases to a plurisubharmonic function, which is not identically zero.
Here $\circ$ denotes the composition. 
The distinctiveness of polynomial skew products  
was pointed out at numerous instances in their papers.

In this paper 
we study the dynamics of polynomial skew products on $\mathbb{C}^2$.
A polynomial skew product is a polynomial map of the form
$f(z,w) = (p(z),q(z,w))$.
We assume that $\delta = \deg p \geq 2$ and $d = \deg_w q \geq 2$.
Let $b(z)$ be the coefficient of $w^d$ and $\gamma$ its degree. 
Many authors have studied the dynamics of regular polynomial skew products
(e.g. \cite{h}, \cite{h2}, \cite{j} and \cite{dh});
$f$ is regular if and only if $\delta = d = \deg q$,
which implies $\gamma = 0$.
We have studied the dynamics of 
nondegenerate polynomial skew products in \cite{u-weight};
we say that $f$ is nondegenerate if $\gamma = 0$.
Besides giving examples of polynomial skew products
whose Green functions are not defined on some curves,
we introduced the weighted Green function $G_f^{\alpha}$ of $f$ in \cite{u-weight},
\[
G_f^{\alpha} (z,w) = \lim_{n \to \infty} 
\frac{1}{{\lambda}^n} \log^{+} |f^n(z,w)|_{\alpha},
\]
where $|(z,w)|_{\alpha} = \max \{ |z|^{\alpha},|w| \}$,
which is defined, continuous and plurisubharmonic on $\mathbb{C}^2$
for a suitable rational number $\alpha \geq 0$.
Moreover, $f$ extends to an algebraically stable map 
on a weighted projective space,
whose dynamics relates to $G_f^{\alpha}$.
Although the dynamics becomes much more difficult 
without the nondegeneracy condition,
the idea of imposing suitable weights is still effective. 

The dynamics of $f$ consists of 
the dynamics of $p$ on the base space and the dynamics on fibers:
$f^n(z,w) = (p^n(z),Q_z^n(w))$, where
$Q_z^n = q_{p^{n-1}(z)} \circ \cdots \circ q_{p(z)} \circ q_{z}$
and $q_z(w) = q(z,w)$.
As the Green function $G_p$ of $p$,
we define the fiberwise Green function of $f$ as
\[
G_z (w) = \lim_{n \to \infty} \frac{1}{d^n} \log^{+} | Q_z^n(w) |.
\]
Since $G_p$ exists on $\mathbb{C}$,
the existence of $G_z$ implies the existence of $G_f$ and $G_f^{\alpha}$. 
Favre and Guedj proved the existence of $G_z$ on $K_p \times \mathbb{C}$
in \cite[Theorem 6.1]{fg},
which is continuous and plurisubharmonic if $b^{-1}(0) \cap K_p = \emptyset$,
and gave examples 
whose fiberwise Green functions are discontinuous over $J_p = \partial K_p$
in \cite[Proposition 6.5]{fg}.
In this paper, we investigate the existence, 
continuity and plurisubharmonicity of $G_z$ on $A_p \times \mathbb{C}$,
where $A_p$ denotes the set of points whose orbits under $p$ tend to infinity.

A summary of our results is as follows.
We replace the definition of $|(z,w)|_{\alpha}$
by $\max \{ |z|^{max \{ \alpha, 0 \}},|w| \}$
or $\max \{ (|z|+1)^{\alpha},|w| \}$.
In the case $\delta \neq d$,
the weighted Green function $G_f^{\alpha}$ is defined on $\mathbb{C}^2$ 
and continuous and plurisubharmonic on $A_p \times \mathbb{C}$
for a suitable rational number $\alpha$, 
which can be negative if $\delta < d$.
If $\delta > d$, then $\alpha$ is nonnegative,
$\limsup_{n \to \infty} \delta^{-n} \log^{+} |Q_z^n| \leq \alpha G_p$
and so $G_f^{\alpha} = \alpha G_p$ on $\mathbb{C}^2$.
Moreover,
if $\alpha = \gamma /(\delta - d)$,
then the limit 
\[
G_z^{\alpha} (w) = \lim_{n \to \infty} \frac{1}{d^n} \log^{+} 
         \left| \dfrac{Q_z^n(w)}{p^n(z)^{\alpha}} \right| 
\]
is defined, continuous and plurisubharmonic on $A_p \times \mathbb{C}$.
If $\delta < d$, 
then $G_z$ coincides with $G_f$ and $G_f^{\alpha}$ on $\mathbb{C}^2$,
and with $G_z^{\alpha}$ on $A_p \times \mathbb{C}$.
Moreover,
if $\delta \neq d$ and $\alpha = \gamma /(\delta - d)$,
then we obtain certain uniform convergence to $G_z^{\alpha}$ 
on $A_p \times \mathbb{C}$
and the asymptotics of $G_z^{\alpha}$ near infinity.
In the case $\delta = d$, the dynamics differs
depending on whether $f$ is nondegenerate.
If $\gamma \neq 0$, then $G_f^{\alpha}$ is defined on $\mathbb{C}^2$
if we admit plus infinity
for a suitable rational number $\alpha$, which can be negative. 
Moreover, the limit
$\lim_{n \to \infty} (n \gamma d^{n-1})^{-1} \log^{+} |Q_z^n(w)|$
is defined and plurisubharmonic on $\mathbb{C}^2$. 
It coincides with $G_p$ on a region in $A_p \times \mathbb{C}$ 
and with $0$ on the complement of the region in $A_p \times \mathbb{C}$.
Furthermore, the limit
\[
G(z,w) = \lim_{n \to \infty} \frac{1}{d^n} \log 
         \left| \dfrac{Q_z^n(w)}{p^n(z)^{n \frac{\gamma}{d}}} \right| 
\]
is defined and plurisubharmonic on $A_p \times \mathbb{C}$
if we admit minus infinity.
It is continuous on the region in $A_p \times \mathbb{C}$. 

Throughout the paper
we make use of the extensions of $f$
to rational maps on weighted projective spaces,
which can be algebraically stable 
if and only if $\gamma = 0$ or $\delta > d$.
Some results on the Green functions 
can be illustrated with these extensions.
In particular,
if $\delta > d$ and $\alpha = \gamma /(\delta - d)$
or if $\delta \leq d$,
then $f \sim (z^{\delta}, z^{\gamma} w^d)$ on a region
in the attracting basin of the indeterminacy point $[0:1:0]$,
which induces the results on the Green functions.
Here the notation $f \sim g$ means that 
the ratios of the first and second components
of the maps tend to $1$, respectively.
Note that  
Guedj pointed out in \cite[Example 3.2]{g} 
that $f$ extends to algebraically stable maps 
on Hirzebruch surfaces if $\delta \leq d$.
Under an additional condition, 
he derived the existence of a weighted Green function of $f$
in \cite[Theorem 4.1]{g},
which is continuous and plurisubharmonic on $\mathbb{C}^2$. 

This paper is organized as follows.   
First, 
we briefly recall the dynamics of polynomial skew products
and give the definition of rational extensions on weighted projective spaces
in Section 2.
In Section 3 we recall the results for the nondegenerate case.
The generalized results 
in the cases $\delta > d$, $\delta < d$ and $\delta = d$
are presented in Sections 4, 5 and 6, respectively.
In these sections we offer two examples:
monomial maps, and skew products that are semiconjugate to polynomial products.
The weight $\alpha$ contributes also to the degree growth of $f$;
in Section 7 we provide a result on the degree growth
and a corollary on a weighted Green function that relates to the degree growth.

\section{Preliminaries}

Let $f(z,w)=(p(z),q(z,w))$ be 
a polynomial skew product such that   
$p(z) = a z^{\delta} + O( z^{\delta - 1} )$ and 
$q(z,w) = b(z) w^{d} + O_z( w^{d - 1} )$,
and let $\gamma = \deg b$.
We assume that $\delta \geq 2$ and $d \geq 2$.
Then we may assume that polynomials $p$ and $b$ are monic
by taking an affine conjugate;
$p(z) = z^{\delta} + O( z^{\delta - 1} )$ and
$b(z) = z^{\gamma} + O( z^{\gamma - 1} )$.
Let $\lambda = \max \{ \delta, d \}$ and 
$\deg f$ be the algebraic degree of $f$,
i.e. $\max \{ \deg p, \deg q \}$.
In general, $\deg f$ may be greater than $\lambda$.
However, the dynamical degree of $f$, 
$\lim_{n \to \infty} \sqrt[n]{\deg (f^n)}$, 
is equal to $\lambda$.

Let us briefly recall  
the dynamics of polynomial skew products.
Roughly speaking, 
the dynamics of $f$ consists of 
the dynamics on the base space and the dynamics on the fibers.
The first component $p$ defines 
the dynamics on the base space $\mathbb{C}$.
Note that $f$ preserves the set of vertical lines in $\mathbb{C}^{2}$.
For this reason, 
we often use the notation $q_z (w)$ instead of $q(z,w)$.
The restriction of $f^n$ to the vertical line $\{ z \} \times \mathbb{C}$
can be viewed as the composition of $n$ polynomials on $\mathbb{C}$, 
$q_{p^{n-1}(z)} \circ \cdots \circ q_{p(z)} \circ q_{z}$.

A useful tool in the study of the dynamics of $p$ on the base space 
is the Green function $G_p$ of $p$, 
\[
G_p(z) = \lim_{n \to \infty} \frac{1}{\delta^n} \log^{+} |p^n(z)|.
\]
It is well known that $G_p$ is defined, 
continuous and subharmonic on $\mathbb{C}$.
More precisely, 
$G_p$ is harmonic and positive on $A_p$ and zero on $K_p$, 
and $G_p(z)= \log |z| + o(1)$ as $z \to \infty$. 
Here $A_p = \{ z :  p^n (z) \to \infty \}$,
$K_p = \{ z : \{ p^n (z) \}_{n \geq 1} \ \text{bounded} \}$
and $A_p \sqcup K_p = \mathbb{C}$.
By definition, $G_p(p(z))= \delta G_p(z)$.
In a similar fashion, 
we consider the fiberwise Green function of $f$, 
\[
G_z(w) = \lim_{n \to \infty} \frac{1}{d^n} \log^{+} |Q_z^n(w)|
\text{ or }
G_z^{\lambda} (w) = \lim_{n \to \infty} \frac{1}{{\lambda}^n} \log^{+} | Q_z^n(w) |,
\]
where 
$Q_z^n  = q_{p^{n-1}(z)} \circ \cdots \circ q_{p(z)} \circ q_{z}$.
By definition, $G_{p(z)}(q_z(w)) = d G_z(w)$ if it exists.
Since the limit $G_p$ exists on $\mathbb{C}$,
the existence of $G_z$ or $G_z^{\lambda}$ implies those of $G_f$ and $G_f^{\alpha}$.
Since the existence of $G_z$ 
on $K_p \times \mathbb{C}$ is proved in \cite{fg},
the remaining problem lies in investigating 
the existence on $A_p \times \mathbb{C}$.
We note that it is unclear even if $f$ is regular. 
For the nondegenerate case,
using an argument in the proof of \cite[Theorem 6.1]{fg},
the author showed the existence of $G_z$ 
on an open subset of $A_p \times \mathbb{C}$ in \cite[Lemma 2.3]{u-sym}
with the assumption $\delta \leq d$,
which was improved in \cite{u-weight}.
In this paper we generalize these results for the case $\gamma \neq 0$.

It is useful to consider the dynamics of the rational extensions of $f$. 
Let $r$ and $s$ be any two positive integers.
The weighted projective space $\mathbb{P} (r,s,1)$ is 
a quotient space of $\mathbb{C}^{3} - \{ O \}$,
\[
\mathbb{P} (r,s,1) = \mathbb{C}^{3} - \{ O \} \ \big/ \sim,
\]
where $(z, w, t) \sim (c^r z, c^s w, c t)$
for any $c$ in $\mathbb{C} - \{ 0 \}$. 
Thus $\mathbb{P} (r,s,1) = \mathbb{C}^{2} \sqcup L_{\infty}$,
where $L_{\infty}$ denotes the line at infinity $\{ t = 0 \}$.
We denote a point in $\mathbb{P} (r,s,1)$
by weighted homogeneous coordinates $[ z : w : t ]$.
It follows that  
$f$ extends to a rational map $\tilde{f}$ on $\mathbb{P} (r,s,1)$,
\[
\tilde{f} [ z : w : t ] 
= \left[ p \left( \frac{z}{t^r} \right) t^{\lambda r} 
: q \left( \frac{z}{t^r}, \frac{w}{t^s} \right) t^{\lambda s} 
: t^{\lambda} \right].
\]
We say that $\tilde{f}$ is algebraically stable
if there is no integer $n$ and no hypersurface $V$
such that ${\tilde{f}}^n (V - I_{{\tilde{f}}^n}) \subset I_{\tilde{f}}$,
where $I_{\tilde{f}}$ denotes the indeterminacy set of $\tilde{f}$.
It is well known that if $\tilde{f}$ is algebraically stable on $\mathbb{P}^2$,
then $\deg (f^n) = (\deg f)^n$.
In the final section
we present a similar claim on $\mathbb{P} (r,s,1)$. 
We define the Fatou set $F_{\tilde{f}}$ of $\tilde{f}$ as 
the maximal open set of $\mathbb{P} (r,s,1)$
where the family of iterates $\{ \tilde{f}^n \}_{n \geq 0}$ is normal.
The Julia set $J_{\tilde{f}}$ of $\tilde{f}$ is defined as 
the complement of the Fatou set of $\tilde{f}$.

\section{Nondegenerate case}

To recall the main results in \cite{u-weight},
we assume that $f$ is nondegenerate throughout this section.
We defined the rational number $\alpha$ as 
\[
\min 
\left\{ l \in \mathbb{Q} \ \Big| 
\begin{array}{lcr}
l \lambda \geq n_j + l m_j
\text{ for any integers $n_j$ and $m_j$} \\ 
\text{s.t. } c_j z^{n_j} w^{m_j} \text{ is a term in } q 
\text{ for some } c_j \neq 0
\end{array} 
\right\}
\]
if $\deg_z q > 0$
and as $0$ if $\deg_z q = 0$.  
Since $q$ has only finitely many terms, one can take the minimum.
Indeed, $\alpha$ is equal to 
\[
\max \left\{ \frac{n_j}{\lambda - m_j} \ \Big| 
\begin{array}{lcr}
\ c_j z^{n_j} w^{m_j} \text{ is a term in } q \\
\text{ with $c_j \neq 0$ and } m_j < \lambda 
\end{array}
\right\}.
\]
Clearly, $\alpha \geq 0$, 
and $\alpha = 0$ if and only if $f$ is a polynomial product.
By definition, 
$\alpha \leq \deg_z q$ and $\alpha < \deg q$.
Moreover, 
$\lambda^n \leq \deg (f^n) \leq 
\max \{ 1, \alpha \} \lambda^n$ for any positive integer $n$.

We define the weight of a monomial $z^n w^m$ mainly as $n + \alpha m$.
The weight of $q$ is defined as the maximum of the weights of all terms in $q$.
Then the weight of $q$ is $\alpha \lambda$
and the weight of $Q_z^n(w)$ is $\alpha \lambda^n$.
Let $h$ be the weighted homogeneous part of $q$ of highest weight $\alpha \lambda$. 
In the case $\delta \leq d$, 
it follows from definition that $h$ contains $w^d$.
On the other hand,
if we replace $\alpha$ 
in the definition of the weight of a monomial
by a positive number which is larger than $\alpha$,
then $h$ coincides with $w^d$,
and if we replace $\alpha$ by a positive number which is smaller than $\alpha$,
then $h$ does not contain $w^d$.

We explain the importance of $\alpha$ in terms of 
the rational extensions of $f$,
assuming that $f$ is not a polynomial product.
We saw that $f$ extends to a rational map $\tilde{f}$
on $\mathbb{P} (r,s,1)$ for any two positive integers $r$ and $s$.
Moreover, it extends to 
a weighted homogeneous polynomial on $\mathbb{C}^{3}$ 
if and only if $s/r \geq \alpha$.
Similarly, 
$\tilde{f}$ is algebraically stable 
if and only if $s/r \geq \alpha$;
if $s/r < \alpha$
then $\tilde{f}$ contracts $L_{\infty} - I_{\tilde{f}}$ 
to the indeterminacy point $p_{\infty} = [0:1:0]$.
The map $\tilde{f}$ is holomorphic 
if and only if $\delta = d$ and $s/r \geq \alpha$.
Note that the holomorphic extensions of polynomial maps
to weighted projective spaces are also mentioned in \cite[Section 5.3]{fj}.

Let us list some more details of $\tilde{f}$ 
on $\mathbb{P} (r,s,1)$ for $s/r \geq \alpha$,
which deffer depending on the magnitude relation of $\delta$ and $d$. 
If $\delta > d$, then $I_{\tilde{f}} = \{ p_{\infty} \}$.
Moreover, 
if $s/r = \alpha$ 
then the dynamics on $L_{\infty} - \{ p_{\infty} \}$ is induced by the polynomial $h$,
and if $s/r > \alpha$ 
then $\tilde{f}$ contracts $L_{\infty} - \{ p_{\infty} \}$ 
to the attracting fixed point $[1:0:0]$.
On the other hand, 
$p_{\infty}$ becomes an attracting fixed point if $\delta \leq d$. 
If $\delta < d$,  
then $\tilde{f}$ contracts $L_{\infty} - I_{\tilde{f}}$ to $p_{\infty}$.
If $\delta = d$, 
then the dynamics on $L_{\infty}$ is determined by $h$.  
Here $h$ is the weighted homogeneous part of $q$ of highest weight $sd/r$, 
which coincides with $w^d$ if $s/r > \alpha$.

Via $\alpha$, 
we obtain the results on the Green functions of $f$.
In the case $\delta > d$,
the dynamics of $\tilde{f}$ described above implies the following upper estimate.

\begin{theorem}[\cite{u-weight}] \label{}
Let $\gamma = 0$.
If $\delta > d$, then 
$\tilde{G}_{z}^{\lambda} \leq \alpha G_p$ 
on $\mathbb{C}^{2}$, where 
$\tilde{G}_{z}^{\lambda} = \limsup_{n \to \infty}$ $\delta^{-n} \log^{+}$ $|Q_z^n|$.
In particular, if $\alpha \leq 1$ then $G_f = G_p$ on $\mathbb{C}^{2}$.
\end{theorem}

Hence if $\delta > d$, then $G_f^{\alpha} = \alpha G_p$ on $\mathbb{C}^{2}$.
In the case $\delta \leq d$,
it follows from definition that $f \sim (z^{\delta}, w^d)$ on $W_R$
and that $f(W_R) \subset W_R$ for large $R > 0$,
where $W_R = \{ |w| > R|z|^{\alpha}, |w| > R^{\alpha + 1} \}$.
This implies that $G_z$ is defined, continuous and pluriharmonic on $W_R$.
Although we use the same notation $W_R$ in the sections below,
the definition of $W_R$ differs depending on whether $f$ is nondegenerate.
Let $A_f = \cup_{n \geq 0} f^{-n} (W_R)$.

\begin{theorem}[\cite{u-weight}]\label{} 
Let $\gamma = 0$.
If $\delta \leq d$, 
then $G_z$ is defined, continuous and pluriharmonic on $A_f$.
Moreover, 
$G_z (w)$ tends 
to $0$ if $\delta < d$ and
to $\alpha G_p (z)$ if $\delta = d$
as $(z,w)$ in $A_f$ tends to $\partial A_f$.
\end{theorem}

Hence if $\delta \leq d$, 
then $G_f$ is defined, continuous and pluriharmonic on $A_f$.
Note that $A_f$ is the restriction of 
the attracting basin of $p_{\infty}$ to $\mathbb{C}^{2}$.
If $\delta < d$, 
then $G_z$ is defined, continuous and plurisubharmonic on $\mathbb{C}^{2}$
and it coincides with both $G_f$ and $G_f^{\alpha}$. 
If $\delta = d$, 
then $G_f^{\alpha}$ is defined, continuous and plurisubharmonic on $\mathbb{C}^{2}$;
roughly speaking, it is the maximum of $\alpha G_p$ and $G_z$.
Moreover,
$G_f^{\alpha}$ determines the Fatou and Julia sets of 
the holomorphic map $\tilde{f}$ on $\mathbb{P} (r,s,1)$, 
where $s/r \geq \alpha$. 

Since two theorems above hold with suitable modifications for any $l \geq \alpha$,
we arrive at the following corollary.

\begin{cor}[\cite{u-weight}]
Let $\gamma = 0$.
For any $l \geq \alpha$, the limit $G_f^{l}$ is defined,
continuous and plurisubharmonic on $\mathbb{C}^{2}$.
\end{cor}

The $\alpha$ is optimal in the sense that
in this theorem $\alpha$ can not be replaced by any smaller number.
See \cite[Remark 2, Examples 5.2 and 5.3]{u-weight} for details.

\section{$\delta > d$}

We first study the case $\delta > d$ toward a generalization of 
the results for nondegenerate polynomial skew products.
Assume that $f$ satisfies the condition $\delta > d$
throughout this section.
Besides showing the same upper estimate of $G_z$ as the nondegenerate case,
we also analyze the new phenomena that does not appear in the nondegenerate case.

The definition of $\alpha$ is the same,
which is positive unless $f$ is a polynomial product.
The map $f$ extends to the rational map $\tilde{f}$ on $\mathbb{P} (r,s,1)$,
which is algebraically stable if $s/r \geq \alpha$.
The dynamics of $\tilde{f}$ implies 
the upper estimate of $G_z$ on $\mathbb{C}^2$,
which induces $G_f^{\alpha} = \alpha G_p$ on $\mathbb{C}^2$.
Moreover,
if $\alpha = \gamma / (\delta - d)$, 
then the dynamics of $f$ is much more understandable;
we show the existence of $G_z^{\alpha}$ on $A_p \times \mathbb{C}$. 
This function relates to the dynamics of $\tilde{f}$,
and induces the existence of $G_z$ and $G_f$ 
on some region in $A_p \times \mathbb{C}$. 

The organization of this section is as follows.
In Section 4.1,
we state the definition of $\alpha$,
whose importance is illustrated 
with the weighted homogeneous part of $q$ or $Q_z^n$,
and with the rational map $\tilde{f}$ on $\mathbb{P} (r,s,1)$.
Then the upper estimate of $G_z$ is described without a complete proof.
In addition,
we present two types of examples of polynomial skew products
whose dynamics is rather understandable.
The first type is monomial maps,
whose Green functions are completely described.
The second type is polynomial skew products
that are semiconjugate to polynomial products,
whose Green functions are well understood.
With the assumption $\alpha = \gamma /(\delta - d)$,
we prove the existence of $G_z^{\alpha}$ in Section 4.2
and the uniform convergence to $G_z^{\alpha}$ in Section 4.3, 
which implies the asymptotics of $G_z^{\alpha}$ near infinity.

\subsection{Weights}

The definition of $\alpha$ is the same as in the nondegenerate case:
\[
\min \left\{ l \in \mathbb{Q} \ \Big|
\begin{array}{lcr}
l \delta \geq n_j + l m_j \text{ for any integers $n_j$ and $m_j$ s.t.} \\ 
z^{n_j} w^{m_j} \text{ is a term in } q \text{ with nonzero coefficient}
\end{array} 
\right\},
\]
which is equal to
\[
\max \left\{ \dfrac{n_j}{\delta - m_j} \ \Big|
\begin{array}{lcr}
z^{n_j} w^{m_j} \text{ is a term in } q \\
\text{with nonzero coefficient}
\end{array} 
\right\}. 
\]
Clearly, $\alpha \geq 0$, 
and $\alpha = 0$ if and only if $f$ is a polynomial product.
By definition, 
$\alpha \leq \deg_z q$ and $\alpha < \deg q$.
Moreover, 
$\delta^n \leq \deg (f^n) \leq \max \{ 1, \alpha \} \delta^n$ 
for any positive integer $n$.

We define the weight of a monomial $z^n w^m$ mainly as $n + \alpha m$.
As same as the nondegenerate case,
the weight of $q$ is $\alpha \delta$
and the weight of $Q_z^n(w)$ is $\alpha \delta^n$.
Let $h$ be the weighted homogeneous part of $q$
of highest weight $\alpha \delta$,
which contains $z^{\gamma} w^d$ if $\alpha = \gamma /(\delta - d)$
and does not if $\alpha > \gamma /(\delta - d)$.
To begin with,
let us consider the case where the rational number $\alpha$ is an integer.
Put $w = cz^{\alpha}$, then 
$h(z, cz^{\alpha}) = h(1,c) z^{\alpha \delta}$.
Fix $c$ so that $q(z,cz^{\alpha})$ can be regarded as a polynomial in $z$.
By letting $h(c) = h(1,c)$, it follows from definition that 
$h(c) z^{\alpha \delta}$ is the homogeneous part of 
$q(z,cz^{\alpha})$ of degree $\alpha \delta$.
Moreover,
$h^n(c) z^{\alpha {\delta}^n}$ is the homogeneous part of 
$Q_z^n (cz^{\alpha})$ of degree $\alpha {\delta}^n$.
Therefore,
$z^{\alpha {\delta}^n} h^n(z^{- \alpha} w)$ is the weighted homogeneous part 
of $Q_z^n (w)$ of weight $\alpha \delta^n$.
If $\alpha$ is not an integer, then
$z^{\alpha}$ is not well defined 
and $c$ is not uniquely determined by $z$ and $w$.
However, in that case,
the polynomial $h$ has some symmetries 
related to the denominator of $\alpha$,
and these notations are still helpful.

The dynamics of $\tilde{f}$ on $\mathbb{P} (r,s,1)$ is also the same as the nondegenerate case.
The details are as follows.
If $s/r < \alpha$, 
then $\tilde{f}$ contracts $L_{\infty} - I_{\tilde{f}}$ 
to the indeterminacy point $p_{\infty} = [0:1:0]$;
thus it is not algebraically stable.
On the other hand,
if $s/r \geq \alpha$, then $\tilde{f}$ is algebraically stable.
If $s/r = \alpha$, 
then $\tilde{f} [z:w:t] = [z^{\delta} + tu(z,t): h(z,w) + tv(z,w,t): t^{\delta}]$,
where $u$ and $v$ are polynomials.
Since $h$ is divisible by $z$,
it follows that $I_{\tilde{f}} = \{ p_{\infty} \}$.
Since $\tilde{f} [z:w:0] = [z^{\delta}: h(z,w): 0]$,
the dynamics of $\tilde{f}$ on $L_{\infty} - \{ p_{\infty} \}$ is induced 
by the dynamics of $h(1,w)$. 
If $s/r > \alpha$, 
then $\tilde{f} [z:w:t] = [z^{\delta} + tu(z,t): tv(z,w,t): t^{\delta}]$.
Hence $I_{\tilde{f}} = \{ p_{\infty} \}$ and 
$\tilde{f}$ contracts $L_{\infty} - I_{\tilde{f}}$ to the attracting fixed point $[1:0:0]$.
Because the fixed point $[1:0:0]$ is attracting in a very strong sense as above,
we obtain the upper estimate
$\tilde{G}_z^{\lambda} \leq l G_p$ on $A_p \times \mathbb{C}$
for any $l = s/r > \alpha$, where $\tilde{G}_z^{\lambda} 
= \limsup_{n \to \infty} \delta^{-n} \log^{+} |Q_z^n|$.
Therefore,

\begin{pro} 
If $\delta > d$, then
$\tilde{G}_z^{\lambda} \leq \alpha G_p$ on $A_p \times \mathbb{C}$.
\end{pro}

One can also prove this proposition along the same line 
as the proof of \cite[Theorem 3.2]{u-weight}.
Hence if $\delta > d$, then
$G_p \leq \tilde{G}_f \leq \max \{ \alpha, 1 \} G_p$ 
on $A_p \times \mathbb{C}$,
where $\tilde{G}_f
= \limsup_{n \to \infty} \delta^{-n} \log^{+} |f^n|$.
In particular, if $\alpha \leq 1$, 
then $G_f = G_p$ on $\mathbb{C}^2$.
In addition,
the existence of $G_z$ on $K_p \times \mathbb{C}$ implies that 
$G_f^{\alpha} = G_f = G_z^{\lambda} = 0$ on $K_p \times \mathbb{C}$,
since $\lambda = \delta > d$.
Consequently,
the proposition above implies the following corollary.

\begin{cor}\label{delta > d; main cor}
If $\delta > d$, then
$G_f^{\alpha} = \alpha G_p$ on $\mathbb{C}^2$.
\end{cor}

We end this subsection with two examples of polynomial skew products 
whose dynamics are well understood:
monomial maps, and skew products that are semiconjugate to polynomial products.

\begin{ex}[monomial maps]
Let $f = (z^{\delta}, z^{\gamma} w^d)$, $\delta > d$ and $\gamma \neq 0$.
Then $\alpha = \gamma /(\delta - d) > 0$ and
$f^n = (z^{\delta^n}, z^{\gamma_n} w^{d^n})$,
where
\[
\gamma_n = (\delta^{n - 1} + \delta^{n - 2} d + \cdots + d^{n - 1}) \gamma
= \alpha {\delta}^n \left\{ 1 - \left( \dfrac{d}{\delta} \right)^n \right\}.
\]
Hence $G_z$ is $\infty$ on $\{ |z| > 1, w \neq 0 \}$, 
$\log^{+} |w|$ on $\{ |z| = 1 \}$, and
$0$ on $\{ |z| < 1 \} \cup \{ w = 0 \}$.
Since $G_p = \log^{+} |z|$, 
\[
G_z^{\lambda} = 
\begin{cases}
\alpha \log^{+} |z| & \text{ on } \{ w \neq 0 \} \\
0 & \text{ on } \{ w = 0 \},
\end{cases}
\] 
which is not continuous on $\{ |z| \geq 1, w = 0 \}$, and
\[
G_f = 
\begin{cases}
\max \left\{ \alpha, 1 \right\} \log^{+} |z| & \text{ on } \{ w \neq 0 \} \\
\log^{+} |z| & \text{ on } \{ w = 0 \},
\end{cases}
\]
which is continuous on $\mathbb{C}^2$ if $\alpha \leq 1$ 
and is not on $\{ |z| \geq 1, w = 0 \}$ if $\alpha > 1$.
Therefore, $G_f^{\alpha} = \alpha \log^{+} |z|$ on $\mathbb{C}^2$,
which is continuous on $\mathbb{C}^2$.
The limits $G_z^{\lambda}$, $G_f$ and $G_f^{\alpha}$
are all plurisubharmonic on $\mathbb{C}^2$.
\end{ex}

\begin{ex}\label{delta > d; ex2} 
Let $f = (z^{\delta}, q(z,w))$ be a polynomial skew product,
where $q = z^{\gamma} w^d + O_z(w^{d-1})$, 
$\gamma \neq 0$ and $\delta > d$, 
that is semiconjugate to $f_0 = (z^{\delta}, h(w))$
by $\pi = (z^r,z^s w)$ for some positive integers $r$ and $s$;
$f \pi = \pi f_0$.
Note that $h(w) = q(1,w)$ and so the degree of $h$ is $d$.
The identity $q(z^r,z^s w) = z^{s \delta} q(1,w)$ implies
that $\alpha = s/r > 0$ and 
$q(z,w) = z^{\alpha \delta} h(w/z^{\alpha})$.
Moreover,
$f^n = (z^{\delta^n}, z^{\alpha \delta^n} h^n (w/z^{\alpha}))$.
Hence 
$w_n/z_n^{\alpha} = h^n (w/z^{\alpha})$ and so
\[
G_z^{\alpha} (w) =
\lim_{n \to \infty} \frac{1}{d^n}
\log^{+} \left| \dfrac{w_n}{z_n^{\alpha}} \right|
= G_h \left( \dfrac{w}{z^{\alpha}} \right) 
\text{ on } \mathbb{C}^{2} - \{ z = 0 \},
\]
where $(z_n,w_n) = f^n(z,w)$. 
Define 
\[
E_f = \bigcup_{|z| > 1} \{ z \} \times z^{\alpha} E_h
\text{ and } 
E_h = \bigcap_{l \geq 0} \overline{ \bigcup_{n \geq l} h^{-n} (0) }
\]
as in \cite{u-semiconj}.
Then 
$G_z^{\lambda} = \alpha \log^{+} |z|$ on $\mathbb{C}^{2} - E_f$,
which implies that 
$G_f = \max \{ \alpha, 1 \} \log^{+} |z|$ on $\mathbb{C}^{2} - E_f$
and that
$G_f^{\alpha} = \alpha \log^{+} |z|$ on $\mathbb{C}^{2}$.
If $0 \not\in E_h$ then the equalities above extend to $\mathbb{C}^{2}$,
and if $\alpha \leq 1$ then $G_f = \log^{+} |z|$ on $\mathbb{C}^{2}$.
\end{ex}

The claims in this example follow from the same line as in \cite{u-semiconj}.
These maps can be characterized by the symmetries of the Julia sets,
see \cite[Theorems 5.2 and 5.5]{u-sym2} for details.

\subsection{Existence of Green functions: $\alpha = \gamma/(\delta - d)$}

We saw that,
for the special map $f$ in Example \ref{delta > d; ex2},
the limit $G_z^{\alpha}$ is defined, 
continuous and plurisubharmonic on $\mathbb{C}^{2} - \{ z = 0 \}$. 
In this subsection,
assuming that $\alpha = \gamma/(\delta - d)$, 
we derive the existence, continuity and plurisubharmonicity 
of $G_z^{\alpha}$ on $A_p \times \mathbb{C}$.
We also assume that $f$ is not a polynomial product
for simplicity,
which implies that $\alpha > 0$ and $\gamma \neq 0$.

Let $W_R = \{ |z| > R, |w| > R|z|^{\alpha} \}$ for large $R > 0$. 
Note that the definition of $W_R$ differs from the nondegenerate case.
We often use the new variety $c = z^{- \alpha} w$.
Although $c$ depends on the choice of the branch of $z^{- \alpha}$,
$|c|$ does not and $W_R = \{ |z| > R, |c| > R \}$.
The following important lemma follows from 
the definition of $\alpha$.

\begin{lem}\label{delta > d: main lem}
If $\delta > d$ and $\alpha = \gamma/(\delta - d)$, then
\[
\left| \dfrac{q(z,w)}{p(z)^{\alpha}} \right| 
\sim \left| \dfrac{w}{z^{\alpha}} \right|^d 
\text{ on } W_R
\]
and $f$ preserves $W_R$; that is, $f(W_R) \subset W_R$.
\end{lem}

\begin{proof}
We explain this claim by using the notation $|c| = |z^{- \alpha} w|$.
Let $z^{n_j} w^{m_j}$ be a term of $q$ with nonzero coefficient.
If $z^{n_j} w^{m_j} \neq z^{\gamma} w^d$, then
\[
\left| \dfrac{z^{n_j} w^{m_j}}{c^d z^{\alpha \delta}} \right| 
= \left| \dfrac{c^{m_j} z^{n_j + \alpha m_j}}{c^d z^{\alpha \delta}} \right| 
\to 0 \text{ as $z$ and $c$} \to \infty 
\] 
since at least one of the inequalities 
$d > m_j$ and $\alpha \delta > n_j + \alpha m_j$ holds.
Thus 
\[
\left| \dfrac{q(z,w)}{z^{\alpha \delta}}  \right| 
= \left| \dfrac{q(z,cz^{\alpha})}{z^{\alpha \delta}}  \right| 
= |c|^d \{ 1 + o(1) \}
= \left| \dfrac{w}{z^{\alpha}} \right|^d \{ 1 + o(1) \}.
\]
Hence there exist positive constants $r_1 < 1 < r_2$ such that
\begin{equation}\label{delta > d: eq1}
r_1 \left| \dfrac{w}{z^{\alpha}} \right|^d 
< \left| \dfrac{q(z,w)}{p(z)^{\alpha}} \right| 
< r_2 \left| \dfrac{w}{z^{\alpha}} \right|^d
\end{equation}
on $W_R$ for large $R$,
since $p(z) \sim z^{\delta}$ as $z \to \infty$.

Let us show that $f$ preserves $W_R$.
It is clear that if $R$ is large enough,
then $|p(z)| > R$
since $p(z) \sim z^{\delta}$ as $z \to \infty$.
Hence it is enough to show that 
$|q(z,w)| > R|p(z)|^{\alpha}$ on $W_R$ for large $R$,
which follows from inequality $(\ref{delta > d: eq1})$
since $d \geq 2$.
\end{proof}

\begin{rem}
We can also show the following
as Lemmas \ref{delta < d: main lem} and \ref{delta = d: main lem}
with a slight change of the proof:
if $\delta > d$ and $\alpha = \gamma/(\delta - d)$, then
$q(z,w) \sim z^{\gamma} w^d$ on $W_R$ for large $R > 0$.
\end{rem}

Let $\tilde{f}$ be the rational extension to $\mathbb{P} (r,s,1)$,
where $s/r = \alpha$. 
Recall that the dynamics of $\tilde{f}$ 
restricted to $L_{\infty} - \{ p_{\infty} \}$
is induced by the polynomial $h(1,w)$ of degree $d$.
Hence $p_{\infty}$ attracts most nearby points in $A_p \times \mathbb{C}$,
and $W_R$ is included in the attracting basin of $p_{\infty}$.
Therefore, 
the inclusion $f(W_R) \subset W_R$ is natural.
Let $A_f = \cup_{n \geq 0} f^{-n} (W_R)$,
which is the restriction of the attracting basin of $p_{\infty}$
to $A_p \times \mathbb{C}$.

\begin{theorem}\label{delta > d: main thm}
If $\delta > d$ and $\alpha = \gamma /(\delta - d)$, then
the limit $G_z^{\alpha}$ is defined, 
continuous and pluriharmonic on $A_f$.
Moreover, 
$G_z^{\alpha} = \log |z^{-\alpha} w| + o(1)$ on $W_R$, 
$G_z^{\alpha} \sim \log |w|$ as $w \to \infty$ for fixed $z$ in $A_p$, 
and $G_z^{\alpha}$ tends to $0$ as $(z,w)$ in $A_f$ 
tends to $\partial A_f - J_p \times \mathbb{C}$.  
\end{theorem}

\begin{proof}
First, 
we prove the uniform convergence of $G_n$ to $G_z^{\alpha}$ on $W_R$,
where $G_n(z,w) = d^{-n} \log | z_n^{- \alpha} w_n |$
and $(z_n,w_n) = f^n (z,w)$.
It follows from inequality $(\ref{delta > d: eq1})$ 
in the proof of Lemma {\rmfamily \ref{delta > d: main lem}} that,
for any $(z,w)$ in $W_R$,
\[
r_1 \left| \dfrac{w_n}{z_n^{\alpha}} \right|^d 
< \left| \dfrac{w_{n+1}}{z_{n+1}^{\alpha}} \right| 
< r_2 \left| \dfrac{w_n}{z_n^{\alpha}} \right|^d. 
\]
Hence, for any $(z,w)$ in $W_R$ and for any positive integer $n$,
\[
\left| G_{n+1} (z,w) - G_n (z,w) \right| 
= \left| \dfrac{1}{d^{n+1}} \log \left| \dfrac{w_{n+1}}{z_{n+1}^{\alpha}} \right| 
\cdot \left| \dfrac{w_{n}}{z_{n}^{\alpha}} \right|^{-d} \right| 
< \dfrac{\log r}{d^{n+1}},
\]
where $\log r = \max \{ - \log r_1, \log r_2 \}$.
Therefore, 
$G_n$ converges uniformly to $G_z^{\alpha}$ on $W_R$.
Since $G_n$ is continuous and pluriharmonic on $W_R$,
the limit $G_z^{\alpha}$ is also continuous and pluriharmonic on $W_R$.
By the inequality above, for any $(z,w)$ in $W_R$,
\begin{eqnarray}\label{delta > d: eq2}
 \begin{split}
 \left| G_z^{\alpha}(w) - \log \left| \dfrac{w}{z^{\alpha}} \right| \right| 
 \leq & \sum_{n=0}^{\infty} \left| G_{n+1} (z,w) - G_n (z,w) \right| \\ 
 < & \sum_{n=0}^{\infty} \dfrac{\log r}{d^{n+1}} = \dfrac{\log r}{d-1} =: C_R.  
 \end{split}
\end{eqnarray}
In particular, 
$G_z^{\alpha} \sim \log |w|$ as $w \to \infty$ for fixed $z$ in $W_R$.

We can extend the domain of $G_z^{\alpha}$ from $W_R$ to $A_f$.
Indeed, for any $(z,w)$ in $A_f$, 
there exists a positive integer $n$ 
such that $f^n (z,w)$ belongs to $W_R$.
Then we define $G_z^{\alpha}(w)$ as $d^{-n} G_{p^n(z)}^{\alpha}(Q_z^n(w))$.
Clearly, $G_z^{\alpha}$ is continuous and pluriharmonic on $A_f$,
and $G_z^{\alpha} \sim \log |w|$ as $w \to \infty$ for fixed $z$ in $A_p$.

Next, we use inequality $(\ref{delta > d: eq2})$ 
to calculate the asymptotic value of $G_z^{\alpha}(w)$ as $(z,w)$ in $A_f$ 
tends to $\partial A_f - J_p \times \mathbb{C}$.
Let $E = \{ |w| = R|z|^{\alpha}, |z| > R \} \subset \partial W_R$.
It then follows from inequality $(\ref{delta > d: eq2})$ that
\[
\left| \ G_{p^n(z)}^{\alpha}(Q_z^n(w)) - \log R \ \right| 
\leq C_R \ \text{ on } f^{-n} (E),
\]
since $f^n (z,w)$ belongs to $E$ 
for any $(z,w)$ in $f^{-n} (E)$.
Thus it follows from equation 
$G_{p^n(z)}^{\alpha}(Q_z^n(w)) = d^n G_z^{\alpha}(w)$
that 
\begin{equation}\label{delta > d: eq3}
\left| \ G_z^{\alpha}(w) - d^{-n} \log R \ \right| \leq d^{-n} C_R
\ \text{ on } f^{-n} (E).
\end{equation}
Since $f^{-n} (E)$ converges to 
$\partial A_f - J_p \times \mathbb{C}$ as $n$ tends to infinity, 
$G_z^{\alpha}(w)$ converges to $0$ 
as $(z,w)$ in $A_f$ tends to $\partial A_f - J_p \times \mathbb{C}$. 
\end{proof}

We remark that the set $\partial A_f - J_p \times \mathbb{C}$
in Theorem {\rmfamily \ref{delta > d: main thm}}
can be replaced by its closure.
Let $B_f = A_p \times \mathbb{C} - A_f$.
Since $G_z^{\alpha} = 0$ on $B_f$,
we get the following corollary.

\begin{cor}
If $\delta > d$ and $\alpha = \gamma /(\delta - d)$, then
the limit $G_z^{\alpha}$ is defined, 
continuous and plurisubharmonic on $A_p \times \mathbb{C}$.
Moreover,
it is pluriharmonic on $A_f$ and int$B_f$.
\end{cor}

Here are some properties of $A_f$ and $B_f$.
First, $A_f$ and $B_f$ are invariant under $f$;
that is, $f(A_f) \subset A_f = f^{-1} (A_f)$
and $f(B_f) \subset B_f = f^{-1} (B_f)$.
Theorem {\rmfamily \ref{delta > d: main thm}} 
guarantees that the set
$\{ \{ z \} \times \mathbb{C} : \deg_w q_z = 0 \}$ of degenerate fibers 
does not intersect with $A_f$.
Hence $B_f$ is not empty; more precisely,
$B_f \cap (\{ z \} \times \mathbb{C}) \neq \emptyset$
for any $z$ in $A_p$.

The existence of $G_z^{\alpha}$ implies 
the following two corollaries on the existence of other Green functions.

\begin{cor}
If $\delta > d$ and $\alpha = \gamma /(\delta - d)$, then
$G_z = \infty$ on $A_f$.
\end{cor}

We note that there exist polynomial skew products such that
$G_z = \infty$ on $A_p \times \mathbb{C}$. 
Indeed,
let $f = (z^{\delta}, q(z,w))$ be a polynomial skew product
that is semiconjugate to $(z^{\delta}, h(w))$
by $\pi = (z^r, z^s w)$.
If $0 \in A_h$, 
then $G_z = \infty$ on $\{ |z| > 1 \}$. 

\begin{cor}
If $\delta > d$ and $\alpha = \gamma /(\delta - d)$, then
$G_z^{\lambda} = \alpha G_p$ and
$G_f = \max \{ \alpha, 1 \} G_p$ on $\mathbb{C}^2 - B_f$.
\end{cor}

In particular, 
if $\delta > d$ and $\alpha = \gamma /(\delta - d)$, 
then $G_z^{\lambda}$ exists on $A_f$.
We can insist on the optimality of $\alpha$ and $A_f$
as in \cite[Remark 2, Examples 5.2 and 5.3]{u-weight},
using polynomial skew products 
that are semiconjugate to polynomial products. 

We end this subsection with a description of the dynamics of 
$\tilde{f}$ on $\mathbb{P} (r,s,1)$,
where $s/r = \alpha$.  
Let $U = \text{int} \overline{A_p \times \mathbb{C}}$,
where the interior and closure are taken in $\mathbb{P} (r,s,1)$.
Then $U = (A_p \times \mathbb{C}) \cup (L_{\infty} - \{ p_{\infty} \})$.
Let $A_{\tilde{f}}$ be the restriction of 
the attracting basin of $p_{\infty}$ to $U$.
It is equal to the union of preimages
$\tilde{f}^{-n} (\text{int} \overline{W_R})$ and hence open.
Moreover,
$A_{\tilde{f}} = A_f \cup A_h$,
where $A_h$ denotes the set of points in $L_{\infty} - \{ p_{\infty} \}$
whose orbits converge to $p_{\infty}$.
Since $A_{\tilde{f}}$ is included in the attracting basin of $p_{\infty}$,
it follows that $A_{\tilde{f}} \subset F_{\tilde{f}}$.
Let $B_{\tilde{f}} = U - A_{\tilde{f}}$.
Then $B_{\tilde{f}} = B_f \cup K_h$,
where $K_h$ denotes the set of points in $L_{\infty} - \{ p_{\infty} \}$
whose orbits do not converge to $p_{\infty}$.
Since int$B_{\tilde{f}} \cap \text{int} \overline{\{ |z| > R \}}$ is Kobayashi hyperbolic
and preserved by $\tilde{f}$,
it follows int$B_{\tilde{f}} \subset F_{\tilde{f}}$
(see \cite{u-weight} for details).

\begin{pro}\label{delta > d: Fatou and Julia}
Let $\delta > d$ and $\alpha = \gamma/(\delta - d)$.
The restriction of $F_{\tilde{f}}$ to $U$ 
consists of $A_{\tilde{f}}$ and int$B_{\tilde{f}}$.
The restriction of $J_{\tilde{f}}$ to $U$ 
is equal to the restriction of $\partial A_{\tilde{f}}$ to $U$,
and to the restriction of $\partial B_{\tilde{f}}$ to $U$.
Moreover, it coincides with the restriction of the closure of
$\{ (z,w) \in A_p \times \mathbb{C} : 
G_z^{\alpha} \text{ is not pluriharmonic} \}$ to $U$.
\end{pro}

The dynamics of $\tilde{f}$ on $B_{\tilde{f}}$ is as follows.
Any point in $B_{\tilde{f}}$ is attracted to $L_{\infty}$ under iteration.
Eventually, the dynamics on $L_{\infty}$,
which is induced by $h$,
should determine the dynamics of $\tilde{f}$ on $B_{\tilde{f}}$.
It follows that $h(c) := h(1,c)$ can be written as $c^{l} H(c^r)$
for some integer $l \geq 0$ and some polynomial $H$,
which is semiconjugate to $c^{l} H(c)^r$ by $c^r$.
The restriction of $\tilde{f}$ to $L_{\infty}$ is conjugate to $c^{l} H(c)^r$.

\subsection{Uniformly convergence and Asymptotics: $\alpha = \gamma/(\delta - d)$}
Continuing with the assumption $\alpha = \gamma/(\delta - d)$,
we show that the convergence to $G_z^{\alpha}$ is uniform 
on some region in $A_p \times \mathbb{C}$,
which induces the asymptotic of $G_z^{\alpha}$ near infinity. 

\begin{pro}\label{delta > d: unif}
If $\delta > d$ and $\alpha = \gamma /(\delta - d)$, then
the convergence to $G_z^{\alpha}$ is uniform on $V \times \mathbb{C}$,
where $\overline{V} \subset A_p$.
\end{pro}

\begin{proof}
We show that, for any $\epsilon > 0$ and 
for any subset $V$ such that $\overline{V} \subset A_p$,
there exists $N$ such that 
$|G_n - G_z^{\alpha}| < \epsilon$ on $V \times \mathbb{C}$ for any $n \geq N$,
where $G_n(z,w) = d^{-n} \log^{+} | z_n^{- \alpha} w_n |$
and $(z_n,w_n) = f^n(z,w)$.
Let $U = \{ z \in V, G_z^{\alpha} \geq \epsilon /3 \}$.
It is enough to show the uniform convergence of $G_n$ to $G_z^{\alpha}$ on $U$.
Indeed, if this holds,
then there exists $N$ such that 
$|G_n - G_z^{\alpha}| < \epsilon /3$ on $U$ for any $n \geq N$.
Hence $G_n < 2 \epsilon /3$ on $\partial U$ for any $n \geq N$.
By Maximum Principle for subharmonic functions on vertical lines,  
$G_n < 2 \epsilon /3$ 
and so $| G_n - G_z^{\alpha} | < \epsilon$ on $U^c$ for any $n \geq N$, 
which completes the proof.

Now we show the uniform convergence on $U$.  
The equation $G_n \circ f = d G_{n+1}$ extends 
the uniform convergent region from $W_R$ to $f^{-n} (W_R)$.
Hence it is enough to show that $U \subset f^{-n} (W_R)$ for large $n$.
From equation $(\ref{delta > d: eq3})$ in the proof of 
Theorem {\rmfamily \ref{delta > d: main thm}}, 
it follows that $G_z^{\alpha} < \epsilon /3$ on $f^{-n} (E)$ for large $n$,
which implies that $U \subset f^{-n} (W_R)$.
\end{proof}

We saw that, for fixed $c = z^{- \alpha} w$,  
the polynomial 
$h^n(c) z^{\alpha d^n}$ is the homogeneous part of 
$Q_z^n (cz^{\alpha})$ of degree $\alpha \delta^n$.
Although $c$ is not well defined if $\alpha$ is not an integer, 
the polynomial $h$ and the Green function $G_h$ 
have some symmetries related to the denominator $r$ of $\alpha$
in that case:
$h(c)$ can be written as $c^{l} H(c^r)$,
the Julia set $J_h$ is preserved by the rotation $\rho c$,
where $\rho$ is a $r$-th root of $1$,
and $G_h (c) = G_h (z^{- \alpha} w)$ is a well defined function in $z$ and $w$.
Hence we get the following asymptotics of $G_z^{\alpha}$ near infinity.

\begin{lem}\label{delta > d: asy lem} 
If $\delta > d$ and $\alpha = \gamma /(\delta - d)$, then
$G_z^{\alpha}(cz^{\alpha}) = G_h(c) + o(1)$ as $z \to \infty$
for fixed $c$. 
\end{lem}

\begin{proof}
We prove that for any $\epsilon > 0$ there exists $R > 0$ 
such that 
\begin{eqnarray}\label{delta > d: eq4}
|G_z^{\alpha}(cz^{\alpha}) - G_h(c)| < \epsilon 
\end{eqnarray}
on $\{ |z| > R \}$.
Assume the rational number $\alpha$ is an integer. 
Since $Q_z^n(cz^{\alpha}) = h^n(c)z^{\alpha {\delta}^n} + o(z^{\alpha {\delta}^n})$,
it follows that
\[
\frac{1}{d^n} \log^{+} \left| \dfrac{Q_z^n(cz^{\alpha})}{z^{\alpha {\delta}^n}} \right| 
= \frac{1}{d^n} \log^{+} |h^n(c)| + o(1).
\] 
Since $p(z) \sim z^{\delta}$ as $z \to \infty$,
\begin{eqnarray}\label{delta > d: eq5}
\frac{1}{d^n} \log^{+} 
\left| \dfrac{Q_z^n(cz^{\alpha})}{p^n(z)^{\alpha}} \right| 
= \frac{1}{d^n} \log^{+} |h^n(c)| + o(1). 
\end{eqnarray}
By Proposition {\rmfamily \ref{delta > d: unif}},
there exist an integer $N_1$ and a number $R > 0$ such that
$|G_z^{\alpha}(w) - G_n(z,w)| < \epsilon /3$ on $\{ |z| > R \}$
for any $n \geq N_1$.
Since the convergence of $G_h^n = d^{-n} \log^{+} |h^n|$ 
to $G_h$ is uniform on $\mathbb{C}$,
there exists $N_2$ such that
$|G_h(c) - G_h^n(c)| < \epsilon /3$ for any $n \geq N_2$.
From equation $(\ref{delta > d: eq5})$
it follows that if $R$ is large enough, then
$|G_N(z,w) - G_h^N(c)| < \epsilon /3$ on $\{ |z| > R \}$,
where $N = \max \{ N_1, N_2 \}$.
Consequently, inequality $(\ref{delta > d: eq4})$ holds on $\{ |z| > R \}$.

Even if $\alpha$ is not an integer, it follows that 
$|Q_z^n(cz^{\alpha})| = 
|h^n(c)||z|^{\alpha {\delta}^n} + o(|z|^{\alpha {\delta}^n})$.
A similar argument as above implies inequality $(\ref{delta > d: eq4})$.
\end{proof}

\begin{pro}\label{delta > d: asy pro}
If $\delta > d$ and $\alpha = \gamma /(\delta - d)$, then
$G_z^{\alpha}(w) = G_h \left( z^{- \alpha} w \right) + o(1)$
as $z \to \infty$. 
\end{pro}

\begin{proof}
We prove that for any $\epsilon > 0$ there exists $R>0$ 
such that 
\begin{eqnarray}\label{delta > d: eq6}
| G_z^{\alpha}(w) - G_h(z^{- \alpha} w) | < \epsilon 
\end{eqnarray}
on $\{ |z| > R \}$.
Separating the region $\{ |z| > R \}$ 
into two open sets whose union contains the region,
we get this inequality. 

First, 
we show that this inequality holds on $W_R$ for large $R$,
where $W_R = \{ |z| > R, |w| > R |z|^{\alpha} \}$.
Inequality $(\ref{delta > d: eq2})$ in the proof of 
Theorem \ref{delta > d: main thm} implies that
$|G_z^{\alpha}(w) - \log|z^{- \alpha} w|| < \epsilon /2$ 
on $W_R$ for large $R$.
Since
$| \log|z^{- \alpha} w| - G_h(z^{- \alpha} w) | < \epsilon /2$ 
on $W_R$ for large $R$,
there exists $R_1 > 0$ such that 
inequality $(\ref{delta > d: eq6})$ holds on $W_{R_1}$.

Next, 
let $V_R = \{ |z| > R, |w| < 2 R_1 |z|^{\alpha} \}$.
From a similar argument as the proof of 
Lemma \ref{delta > d: asy lem},
it follows that
$G_z^{\alpha}(w) = G_h (z^{- \alpha} w) + o(1) \text{ on } V_R$,
since the projection of $V_R$ to $\mathbb{C}$
by the multi-valued function $z^{- \alpha} w$
is a relatively compact subset of $\mathbb{C}$.
Hence there exists $R_2 \geq R_1$ such that 
inequality $(\ref{delta > d: eq6})$ holds on $V_{R_2}$.
Since $W_{R_2}$ and $V_{R_2}$ cover  
the region $\{ |z| > R_2 \}$,
inequality $(\ref{delta > d: eq6})$ holds on 
$\{ |z| > R_2 \}$.
\end{proof}

\section{$\delta < d$}

Next we study the case $\delta < d$,
assuming that $\gamma \neq 0$.
We generalize the definition of $\alpha$,
and prove the existence 
of $G_z$ on $\mathbb{C}^2$ as the nondegenerate case,
which is continuous and plurisubharmonic on $A_p \times \mathbb{C}$.
Moreover, 
$G_z = G_f = G_f^{\alpha}$ on $\mathbb{C}^2$
and $G_z = G_z^{\alpha}$ on $A_p \times \mathbb{C}$.
Unlike the nondegenerate case,
$\alpha$ can be negative and 
$\tilde{f}$ is not algebraically stable unless $f$ is nondegenerate.
However, the indeterminacy point $p_{\infty}$ 
is still attracting in some sense,
and it follows that $f \sim (z^{\delta}, z^{\gamma} w^d)$
on a region in the attracting basin of $p_{\infty}$,
which induces the results above.
To obtain uniform convergence to $G_z$,
we have to consider unusual plurisubharmonic functions
that converge to $G_z$.
In particular,
if $\alpha = \gamma /(\delta - d)$,
then the convergence to $G_z^{\alpha}$ is uniform 
on a region in $A_p \times \mathbb{C}$,
and we get the asymptotics of $G_z^{\alpha}$ near infinity.   

This section is divided into three subsections:
the definition and the properties of $\alpha$,
the existence of $G_z$ and other Green functions,
and the uniform convergence to $G_z$. 
We remark that many claims hold even if $\gamma = 0$.

\subsection{Weights}

We generalize the definition of $\alpha$ as
\[
\min \left\{ l \in \mathbb{Q} \ \Big| 
\begin{array}{lcr}
\gamma + ld \geq l \delta \text{ and } \gamma + ld \geq n_j + l m_j \text{ for} \\
\text{any integers $n_j$ and $m_j$ s.t. } z^{n_j} w^{m_j} \text{ is} \\
\text{a term in $q$ with nonzero coefficient}
\end{array} 
\right\},
\]
which is equal to
\[
\max \left\{ \dfrac{-\gamma}{d - \delta}, 
\dfrac{n_j - \gamma}{d - m_j} \ \Big|
\begin{array}{lcr}
z^{n_j} w^{m_j} \text{ is a term in } q \text{ with} \\
\text{nonzero coefficient s.t. } m_j < d
\end{array} 
\right\}.
\]
By definition,
$- \gamma \leq \alpha \leq \deg_z q - \gamma$
and $\alpha < \deg q - \gamma$. 
Moreover, $\alpha$ induces inequalities for the degree growth of $f$,
given in Section 7. 

Let the weight of $z^n w^m$ be $n + \alpha m$.
Then the weight of $q$ is $\gamma + \alpha d$ and 
the weight of $Q_z^n(w)$ is $\gamma_n + \alpha d^n$,
where $\gamma_n =(\delta^{n - 1} + \delta^{n - 2} d + \cdots + d^{n - 1}) \gamma$,
which coincides with $\alpha \delta^n$ if $\alpha = \gamma /(\delta - d)$. 
Let $h$ be the weighted homogeneous part of $q$ 
of highest weight $\gamma + \alpha d$,
which contains $z^{\delta} w^d$.
If $\alpha = \gamma /(\delta - d)$, then
$z^{\alpha {\delta}^n} h^n(z^{- \alpha} w)$ is the weighted homogeneous part of 
$Q_z^n (w)$ of weight $\alpha {\delta}^n$.
If $\alpha > \gamma /(\delta - d)$, then
$z^{\gamma_n} (z^{- \gamma} h(z,w))^{d^{n-1}}$ is the weighted homogeneous part of 
$Q_z^n (w)$ of weight $\gamma_n + \alpha d^n$.

The rational map $\tilde{f}$ on $\mathbb{P} (r,s,1)$ is not algebraically stable
for any positive integers $r$ and $s$;
it contracts $L_{\infty} - I_{\tilde{f}}$ 
to the indeterminacy point $p_{\infty}$. 
However, 
$p_{\infty}$ attracts most points in $A_p \times \mathbb{C}$;
more precisely, it attracts all points
which converge to $(L_{\infty} - I_{\tilde{f}}) \cup \{ p_{\infty} \}$.

As in Section 4.1,
we end this subsection with two examples:
monomial maps, and polynomial skew products 
that are semiconjugate to polynomial products. 

\begin{ex}[monomial maps]
Let $f = (z^{\delta}, z^{\gamma} w^d)$ and $\delta < d$.
Then $\alpha = \gamma /(\delta - d) \leq 0$ 
and $f^n = (z^{\delta^n}, z^{\gamma_n} w^{d^n})$,
where
\[
\gamma_n = (\delta^{n - 1} + \delta^{n - 2} d + \cdots + d^{n - 1}) \gamma
= - \alpha d^n \left\{ 1 - \left( \dfrac{\delta}{d} \right)^n \right\}
\]
Hence $G_z = G_z^{\alpha} = G_f = G_f^{\alpha} = \log^{+} |z^{- \alpha} w|$,
which is continuous and plurisubharmonic on $\mathbb{C}^{2}$.
\end{ex}

\begin{ex} 
Let $f = (z^{\delta}, q(z,w))$ be a polynomial skew product,
where $q = z^{\gamma} w^d + O_z(w^{d-1})$, $\gamma \neq 0$ and $\delta < d$, 
that is semiconjugate to a polynomial product $f_0 = (z^{\delta}, h(w))$
by $\pi = (z^r, z^{-s} w)$ for some positive integers $r$ and $s$;
$f \pi = \pi f_0$.
Note that $h(w) = q(1,w)$, 
which is divisible by $w$.
It follows from the identity $q(z^r, z^{-s} w) = z^{-s \delta} q(1,w)$
that $\alpha = - s/r < 0$ and 
$q(z,w) = z^{- \alpha \delta} h(z^{- \alpha} w)$.
Moreover,
$f^n = ( z^{\delta^n}, z^{- \alpha \delta^n} h^n (z^{- \alpha} w))$.
Hence $z_n^{- \alpha} w_n = h^n (z^{- \alpha} w)$ and so
\[
G_z^{\alpha}(w) =
\lim_{n \to \infty} \frac{1}{d^n}
\log^{+} \left| z_n^{- \alpha} w_n \right|
= G_h \left( z^{- \alpha} w \right)
\text{ on } \mathbb{C}^{2}, 
\]
where $(z_n, w_n) = f^n (z, w)$.
Moreover,
$G_z = G_z^{\alpha} = G_f = G_f^{\alpha}$,
which is continuous and plurisubharmonic on $\mathbb{C}^{2}$.
\end{ex}

\subsection{Existence of Green functions}

We defined $W_R$ 
as $\{ |z| > R, |w| > R|z|^{\alpha} \}$ if $\gamma \neq 0$. 
The following important lemma follows from 
the definition of $\alpha$.

\begin{lem}\label{delta < d: main lem}
If $\delta < d$, then
$q(z,w) \sim z^{\gamma} w^d$ on $W_R$  for large $R > 0$,
and $f$ preserves $W_R$; that is, $f(W_R) \subset W_R$.
\end{lem}

\begin{proof}
We explain this claim by using the notation $|c| = |z^{- \alpha} w|$.
Let $z^{n_j} w^{m_j}$ be a term of $q$ with nonzero coefficient.
If $z^{n_j} w^{m_j} \neq z^{\gamma} w^d$, then
\[
\left| \dfrac{z^{n_j} w^{m_j}}{z^{\gamma} w^d} \right| 
= \left| \dfrac{c^{m_j} z^{n_j + \alpha m_j}}{c^d z^{\gamma + \alpha d}} \right| 
\to 0 \text{ as $z$ and $c$} \to \infty 
\] 
since at least one of the inequalities 
$d > m_j$ and $\gamma + \alpha d > n_j + \alpha m_j$ holds.
Therefore, $q(z,w) \sim z^{\gamma} w^d$ on $W_R$.
In other word, 
\begin{equation}\label{delta < d: eq1}
r_1 < \left| \dfrac{q(z,w)}{z^{\gamma} w^d} \right| < r_2 
\end{equation}
on $W_R$ for some positive constants $r_1 < 1 < r_2$.

Let us show that $f$ preserves $W_R$.
It is enough to show that 
$|q(z,w)| > R|p(z)|^{\alpha}$ on $W_R$ for large $R$.
Since $p(z) \sim z^{\delta}$,
\[
\left| \dfrac{q(z,w)}{p(z)^{\alpha}} \right| 
= \left| \dfrac{z^{\delta}}{p(z)} \right|^{\alpha} 
\cdot \left| \dfrac{z^{\gamma + \alpha d}}{z^{\alpha \delta}} \right| 
\cdot \left| \dfrac{q(z,w)}{z^{\gamma + \alpha d}} \right| 
\sim \left| \dfrac{z^{\gamma + \alpha d}}{z^{\alpha \delta}} \right| \cdot |c|^d.
\]
This is larger than $R$ if it is large enough
since $\gamma + \alpha d \geq \alpha \delta$ and $d \geq 2$.
\end{proof}

This lemma implies that 
$W_R$ is included in the attracting basin of $p_{\infty}$.
Let $A_f$ be the union of preimages $f^{-n} (W_R)$. 

\begin{theorem}\label{delta < d: main thm}
If $\delta < d$, then
the limit $G_z$ is defined, continuous and pluriharmonic on $A_f$.
Moreover, 
$G_z \sim \log |w|$ as $w \to \infty$ for fixed $z$ in $A_p$, 
and $G_z$ tends to $0$ as $(z,w)$ in $A_f$ 
tends to $\partial A_f - J_p \times \mathbb{C}$.  
\end{theorem}

\begin{proof}
The proof is similar to that of Theorem {\rmfamily \ref{delta > d: main thm}}.
We first prove the locally uniform convergence of $G_n$ to $G_z$ on $W_R$,
where $G_n = d^{-n} \log |Q_z^n|$.
It follows from inequality $(\ref{delta < d: eq1})$ 
in the proof of Lemma {\rmfamily \ref{delta < d: main lem}} that,
for any $(z,w)$ in $W_R$,
\[
r_1 |(p^n(z))^{\gamma}| 
< \left| \dfrac{Q_z^{n+1} (w)}{Q_z^n(w)^{d}} \right| 
< r_2 |(p^n(z))^{\gamma}|. 
\]
Hence, for any $(z,w)$ in $W_R$ and for any positive integer $n$,
\[
\left| G_{n+1} - G_n \right| 
= \left| \dfrac{1}{d^{n+1}} \log \dfrac{|Q_z^{n+1}(w)|}{|Q_z^n(w)|^d}  \right|
< \dfrac{1}{d^{n+1}} \log \left( r|p^n(z)|^{\gamma} \right),
\]
where $\log r = \max \{ - \log r_1, \log r_2 \}$.
Therefore, $G_n$ converges locally uniformly to $G_z$ on $W_R$,
which is continuous and pluriharmonic. 
By the inequality above, for any $(z,w)$ in $W_R$,
\[
| G_z(w) - \log |w| | < \sum_{n=0}^{\infty} 
\dfrac{\gamma}{d^{n+1}} \log |p^n(z)| + \dfrac{\log r}{d - 1}.  
\]
Although this infinite sum converges since $\deg p = \delta < d$,
we can restate this inequality to a simpler form.
Since $p(z) \sim z^{\delta}$,
there exists a constant $r_0 > 1$ such that
$|p(z)| < r_0 |z|^{\delta}$ and so $|p^n(z)| < (r_0 |z|)^{\delta^n}$
if $|z| > R$.
Hence
\begin{equation}\label{delta < d: eq2}
\left| G_z(w) - \log |w| \right| < \dfrac{\gamma}{d - \delta} \log |z| + C_R
\end{equation}
for any $(z,w)$ in $W_R$,
where 
\[
C_R = \dfrac{\gamma}{d - \delta} \log r_0 + \dfrac{\log r}{d - 1}.
\]
In particular,
$G_z \sim \log |w|$ as $w \to \infty$ for fixed $z$ in $W_R$.
We can extend the domain of $G_z$ from $W_R$ to $A_f$ 
by the equation $G_z(w) = d^{-n} G_{p^n(z)} (Q_z^n(w))$.

Next, we 
calculate the asymptotic value of $G_z(w)$ as $(z,w)$ in $A_f$ 
tends to $\partial A_f - J_p \times \mathbb{C}$.
By inequality $(\ref{delta < d: eq2})$, 
\[
\left| G_{p^n(z)} (Q_z^n(w)) - \log R |p^n(z)|^{\alpha} \right| 
\leq \frac{\gamma}{d - \delta} \log |p^n(z)| + C_R 
\]
for any $(z,w)$ in $f^{-n} (E)$,
where $E = \{ |w| = R|z|^{\alpha}, |z| > R \} \subset \partial W_R$.
Thus it follows from the equation $G_z(w) = d^{-n} G_{p^n(z)} (Q_z^n(w))$
that 
\begin{equation}\label{delta < d: eq3}
\left| G_{z} (w) - \dfrac{1}{d^{n}} \log R |p^n(z)|^{\alpha} \right| 
\leq \dfrac{1}{d^{n}} 
\left( \frac{\gamma}{d - \delta} \log |p^n(z)| + C_R \right) 
\end{equation} 
on $f^{-n} (E)$.
Since $\deg p = \delta < d$
and since $f^{-n} (E)$ converges to 
$\partial A_f - J_p \times \mathbb{C}$ as $n$ tends to infinity, 
$G_z(w)$ converges to $0$ 
as $(z,w)$ in $A_f$ tends to $\partial A_f - J_p \times \mathbb{C}$. 
\end{proof}

Moreover, it follows that
$G_z = \log |z^{\gamma /(d - \delta)} w| + o(1)$ on $W_R$;
see Remark {\rmfamily \ref{delta < d: asy of G on W_R} in the next subsection.
Since $G_z = 0$ on $B_f$, 
the theorem above implies the following.

\begin{cor}\label{delta < d; main cor}
If $\delta < d$, then
$G_z$ is defined on $\mathbb{C}^2$,
which is continuous and plurisubharmonic on $A_p \times \mathbb{C}$.
Moreover,
$G_z = G_f = G_f^{\alpha}$ on $\mathbb{C}^2$,
and $G_z = G_z^{\alpha}$ on $A_p \times \mathbb{C}$.
\end{cor}

We end this subsection with a description of 
the dynamics of $\tilde{f}$ on $\mathbb{P} (r,s,1)$ 
as in Section 4.2.
By definition,
\[
\tilde{f} [z:w:t] = 
\left[ t^{r(d - \delta)} \{ z^{\delta} + tu(z,t) \} : 
       \dfrac{h(z,w) + tv(z,w,t)}{t^{r \gamma}}: t^d \right]
\]
\[
= \left[ z^{\delta} + tu(z,t) : 
\dfrac{h(z,w) + tv(z,w,t)}{t^{r \gamma + sd - s \delta}} : t^{\delta} \right]
\]
\[
= [ t^{\frac{r}{s}A} \{ z^{\delta} + tu(z,t) \} :
h(z,w) + tv(z,w,t) : t^{\frac{1}{s}A + \delta} ],
\]
where $u$ and $v$ are polynomials,
$h$ is the weighted homogeneous part of $q$ of weight $sd/r$,
and $A = r \gamma + sd - s \delta > 0$.
The polynomial $h$ contains $z^{\gamma} w^d$
if and only if $s/r \geq \alpha$,
which coincides with $z^{\gamma} w^d$ if $s/r > \alpha$.
Since $I_{\tilde{f}} = \{ [z:w:0] : h(z,w) = 0 \}$ 
and $h$ is divisible by $z$,
the point $p_{\infty} = [0:1:0]$ is always an indeterminacy point.
Note that $\tilde{f}$ has other indeterminacy points 
than $p_{\infty}$ if $s/r \geq \alpha$.
In particular, 
$I_{\tilde{f}} = \{ p_{\infty}, [1:0:0] \}$ if $s/r > \alpha$.

Recall that $p_{\infty}$ attracts all points
which converge to $(L_{\infty} - I_{\tilde{f}}) \cup \{ p_{\infty} \}$.
For the sets $A_{\tilde{f}}$ and $B_{\tilde{f}}$ as before,
it follows that
$A_{\tilde{f}} = A_f \cup (L_{\infty} - I_{\tilde{f}})$
and $B_{\tilde{f}} = B_f \cup (I_{\tilde{f}} - \{ p_{\infty} \})$.
Since $A_{\tilde{f}}$ is included in the attracting basin of $p_{\infty}$ 
and since $B_{\tilde{f}}$ is the attracting basin of $I_{\tilde{f}} - \{ p_{\infty} \}$,
it follows that $A_{\tilde{f}} \cup \text{int} B_{\tilde{f}} \subset F_{\tilde{f}}$.
Hence the same claim as
Proposition {\rmfamily \ref{delta > d: Fatou and Julia}} holds.

\subsection{Uniformly convergence and Asymptotics}

It seems impossible to prove that 
the convergence of $G_n$ to $G_z$ is uniform,
where $G_n = d^{-n} \log^{+} |Q_z^n|$.
To overcome this problem,
we define
\[
\tilde{G}_n(z,w) = G_n(z,w) 
+ \sum_{j = n}^{\infty} \dfrac{\gamma}{d^{j+1}} \log^{+} |p^j (z)|, 
\]
which converges to $G_z$ on $\mathbb{C}^2$. 

\begin{pro}\label{delta < d: unif}
If $\delta < d$, then
the convergence of $\tilde{G}_n$ to $G_z$ is 
uniform on $V \times \mathbb{C}$,
where $\overline{V} \subset A_p$.
\end{pro}

\begin{proof}
We first show that 
$\tilde{G}_n$ converges uniformly to $G_z$ on $W_R$.
By 
Lemma {\rmfamily \ref{delta < d: main lem}},
there is a constant $r > 1$ such that, 
for any $(z,w)$ in $W_R$,
\[
\left| \tilde{G}_{n+1} - \tilde{G}_n \right| 
= \left| G_{n+1} - G_n - \dfrac{\gamma}{d^{n+1}} \log |p^n(z)| \right| 
\]
\[
= \dfrac{1}{d^{n+1}} \log \left| \dfrac{Q_z^{n+1}(w)}{(p^n(z))^{\gamma} (Q_z^n(w))^d}  \right| 
< \dfrac{\log r}{d^{n+1}}.
\]
Hence $\tilde{G}_n$ converges uniformly to $G_z$ on $W_R$.

The left part of the proof is the same as the proof of 
Proposition {\rmfamily \ref{delta > d: unif}}.
The equation $\tilde{G}_{n} \circ f = d \tilde{G}_{n+1}$
extends the uniform convergent region from $W_R$ to $f^{-n} (W_R)$.
It follows from equation $(\ref{delta < d: eq3})$
in the proof of Theorem {\rmfamily \ref{delta < d: main thm}} 
that $U = \{ z \in V, G_z \geq \epsilon /3 \} \subset f^{-n} (W_R)$
for large $n$,
which induces the required uniform convergence.
\end{proof}

This proposition holds even if we replace $\log^{+} |p^j|$ 
in the definition of $\tilde{G}_n$ by $\log |p^j|$,
since $|p^j| > 1$ on $V$ if $j$ is large enough.
Moreover,
in Proposition {\rmfamily \ref{delta < d: unif}}
we can replace $\tilde{G}_n$ by
\[
\hat{G}_n 
= G_n + \sum_{j = n}^{\infty} \dfrac{\gamma}{d^{j+1}} \log |z^{\delta^j}|
= G_n + \dfrac{\gamma}{d - \delta} \left( \dfrac{\delta}{d} \right)^n \log |z| 
\]
with a slight modification of the convergent region.
However, 
the proof is not the same as before, because
\[
\hat{G}_n \circ f - d \hat{G}_{n+1}
= \dfrac{\gamma}{d - \delta} \left( \dfrac{\delta}{d} \right)^n 
\log \left| \dfrac{p(z)}{z^{\delta}} \right|,
\]
which is not zero in general.

\begin{pro}\label{delta < d: unif2}
If $\delta < d$, then
the convergence of $\hat{G}_n$ to $G_z$ is uniform on $V \times \mathbb{C}$, 
where $\overline{V} \subset A_p \cap \{ |z| > 1 \}$.
\end{pro}

\begin{proof}
We first show that 
$\hat{G}_n$ converges uniformly to $G_z$ on $W_R$.
Since $p(z) \sim z^{\delta}$,
there exists constants $0 < r_1 < 1 < r_2$ such that
$r_1 |z^{\delta}| < |p(z)| < r_2 |z^{\delta}|$ and so 
$(r_1 |z|)^{\delta^n} < |p^n(z)| < (r_2 |z|)^{\delta^n}$ if $|z| > R$.
Hence if $|z| > R$, then
\[
\left| \log \left| \dfrac{p^n(z)}{z^{\delta^n}} \right| \right|
< \log r_0^{\delta^n} = \delta^n \log r_0,
\]
where $\log r_0 = \max \{ - \log r_1, \log r_2 \}$.
By Lemma {\rmfamily \ref{delta < d: main lem}},
there is a constant $r > 1$ such that,
for any $(z,w)$ in $W_R$,
\[
\left| \log \left| \dfrac{Q_z^{n+1}(w)}{(p^n(z))^{\gamma} (Q_z^n(w))^d} \right| \right|
< \log r.
\]
With these inequalities, the equation
\[
\left| \dfrac{Q_z^{n+1}(w)}{(z^{\delta^n})^{\gamma} (Q_z^n(w))^d} \right| 
= \left| \dfrac{p^n(z)}{z^{\delta^n}} \right|^{\gamma} 
\cdot \left| \dfrac{Q_z^{n+1}(w)}{(p^n(z))^{\gamma} (Q_z^n(w))^d} \right| 
\]
implies that, for any $(z,w)$ in $W_R$,
\begin{eqnarray}\label{delta < d: eq8} 
\begin{split}
\left| \hat{G}_{n+1} - \hat{G}_n \right| 
&= \left| G_{n+1} - G_n - \dfrac{\gamma}{d^{n+1}} \log |z^{\delta^n}| \right| \\
&= \left| \dfrac{1}{d^{n+1}} \log \left| 
  \dfrac{Q_z^{n+1}(w)}{(z^{\delta^n})^{\gamma} (Q_z^n(w))^d} \right| \right| \\
&< \dfrac{\gamma}{d} \left( \dfrac{\delta}{d} \right)^n \log r_0
+ \dfrac{1}{d^{n+1}} \log r.
\end{split}
\end{eqnarray}
Hence $\hat{G}_n$ converges uniformly to $G_z$ on $W_R$.

Next, we show the uniform convergence 
on $f^{-N} (W_R) \cap ( p_0^{-N} (U) \times \mathbb{C} )$
for any $N$, 
where $p_0 (z) = z^{\delta}$ and $U = \{ |z| > R \}$.
By definition,
there exist constants $0 < m_N < 1 < M_N$ such that
$m_N |z^{\delta^N}| < |p^N (z)| < M_N |z^{\delta^N}|$ 
on $U_N = p^{-N} (U) \cap p_0^{-N} (U)$. 
Since 
\[
\left| \dfrac{p^n(z)}{z^{\delta^n}} \right| 
= \left| \dfrac{p^{n-N}(p^N(z))}{p_0^{n-N}(p^N(z))}
\cdot \dfrac{p_0^{n-N}(p^N(z))}{p_0^{n-N}(p_0^N(z))} \right|,
\]
it follows that,
for any $z$ in $U_N$ and for any $n \geq N$,
\[
(r_1 m_N)^{\delta^{n-N}} 
< \left| \dfrac{p^n(z)}{z^{\delta^n}} \right| 
< (r_2 M_N)^{\delta^{n-N}}. 
\]
Applying a similar argument as above for $n \geq N$,
we get the uniform convergence 
on the required region,
which converges to $A_f \cap ( \{ |z| > 1 \} \times \mathbb{C} )$.
The left part is the same as the proof of
Proposition {\rmfamily \ref{delta > d: unif}}. 
\end{proof}

\begin{rem}\label{delta < d: asy of G on W_R} 
It follows from inequality $(\ref{delta < d: eq8})$ that
\[
\left| G_z(w) - \log |z^{\gamma /(d - \delta)} w| \right| < C_R
\text{ on } W_R,
\]
where $C_R$ converges to $0$ as $R$ tends to $\infty$.
This inequality is better than
inequality $(\ref{delta < d: eq2})$ 
in the proof of Theorem {\rmfamily \ref{delta < d: main thm}}
and resembles inequality $(\ref{delta > d: eq2})$ 
in the proof of Theorem {\rmfamily \ref{delta > d: main thm}}.
\end{rem}

Moreover, if $\alpha = \gamma /(\delta - d) < 0$, then
we can show that the convergence to $G_z^{\alpha}$ is uniform.
Let $\bar{G}_n = d^{-n} \log^{+} |z_n^{- \alpha} w_n|$,
where $(z_n,w_n) = f^n(z,w)$.  
If $\alpha = \gamma /(\delta - d) < 0$, then
$\hat{G}_n = d^{-n} \log |z^{- \alpha \delta^n} w_n|$ and so
\[
\bar{G}_n
= \hat{G}_n + \frac{- \alpha}{d^n} \log \left| \dfrac{p^n(z)}{z^{\delta^n}} \right|
\text{ on } W_R. 
\]
Hence we get the following convergence theorem.

\begin{cor}
If $\delta < d$ and $\alpha = \gamma /(\delta - d) < 0$, then
the convergence of $\bar{G}_n$ to $G_z^{\alpha}$ is uniform on $V \times \mathbb{C}$, 
where $\overline{V} \subset A_p$.
\end{cor}

\begin{proof}
Since $\hat{G}_n$ converges uniformly to $G_z^{\alpha}$ on $W_R$,
so does $\bar{G}_n$.
The equation $\bar{G}_n \circ f = d \bar{G}_{n+1}$
extends the uniform convergent region from $W_R$ to $f^{-n} (W_R)$.  
The left part is the same as the proof of
Proposition {\rmfamily \ref{delta > d: unif}}. 
\end{proof}

This uniform convergence induces 
the following asymptotics of $G_z^{\alpha}$ near infinity.
The proofs are similar to those of 
Lemma {\rmfamily \ref{delta > d: asy lem}}
and Proposition {\rmfamily \ref{delta > d: asy pro}}.

\begin{lem}
If $\delta < d$ and $\alpha = \gamma /(\delta - d)$,
then $G_z^{\alpha}(cz^{\alpha}) = G_h (c) + o(1)$
as $z \to \infty$ for fixed $c$.
\end{lem}

\begin{proof}
Because
$Q_z^n (cz^{\alpha}) = z^{\alpha \delta^n} h^n(c) \{ 1 + o(1) \}$
as $z \to \infty$, 
\[
\hat{G}_n (z, cz^{\alpha}) 
= d^{-n} \log^{+} |h^n(c)| + o(1).
\]
Since $\hat{G}_n$ and $d^{-n} \log^{+} |h^n|$ converge uniformly 
to $G_z^{\alpha}$ and $G_h$ on suitable sets respectively, 
we get the required asymptotics.
\end{proof}

In this proof we can replace $\hat{G}_n$ by $\bar{G}_n$.

\begin{pro}\label{delta < d: asy}
If $\delta < d$ and $\alpha = \gamma /(\delta - d)$, 
then $G_z^{\alpha}(w) = G_h (z^{- \alpha} w) + o(1)$
as $z \to \infty$.
\end{pro}

\section{$\delta = d$}

In this section we deal with the last case $\delta = d$.
The results for the Green functions of $f$
and the dynamics of $\tilde{f}$ are different
depending on whether $f$ is nondegenerate. 
However, it is common that $p_{\infty}$ is attracting in some sense
and that $f \sim (z^{d}, z^{\gamma} w^d)$
on a region in the attracting basin of $p_{\infty}$.

In Section 6.1,
we give the definition of $\alpha$ and an example of monomial maps.
In Section 6.2,
we prove the existence of three types of Green functions
under the assumption $\gamma \neq 0$.
First,
we show that $G_z$ is defined on $A_f$
and $G_f^{\alpha}$ is defined on $\mathbb{C}^2$
if we admit plus infinity.
Next,
we show the existence of 
$\lim_{n \to \infty} (n \gamma d^{n-1})^{-1} \log^{+} |Q_z^n(w)|$
on $\mathbb{C}^2$,
which is continuous on $\mathbb{C}^2 - \partial A_f \cap \partial B_f$
and 
plurisubharmonic on $\mathbb{C}^2$. 
Finally, 
we show that the limit $G$ is defined and plurisubharmonic 
on $A_p \times \mathbb{C}$ if we admit minus infinity.
It is continuous and pluriharmonic on $A_f$. 

\subsection{Weights}

We generalize the definition of $\alpha$ as
\[
\inf \left\{ l \in \mathbb{Q} \ \Big| 
\begin{array}{lcr}
\gamma + ld \geq n_j + l m_j \text{ for any integers $n_j$ and $m_j$ s.t.} \\ 
z^{n_j} w^{m_j} \text{ is a term in } q \text{ with nonzero coefficient}
\end{array}
\right\}.
\] 
This definition is similar to the previous case $\delta < d$,
since the inequality $\gamma + ld \geq l \delta$ is trivial if $\delta = d$.
If $q(z,w) \neq b(z) w^d$, then
we can replace the infimum in the definition of $\alpha $ by the minimum,
which is equal to
\[
\max \left\{ \dfrac{n_j - \gamma}{d - m_j} \ \Big|
\begin{array}{lcr}
z^{n_j} w^{m_j} \text{ is a term in } q \text{ with} \\
\text{nonzero coefficient s.t. } m_j < d
\end{array}
\right\}.
\] 
For this case,
$- \gamma \leq \alpha \leq \deg_z q - \gamma$
and $\alpha < \deg q - \gamma$.
If $q(z,w) = b(z) w^d$, then $\alpha = - \infty$.
Thus $\alpha = - \infty$ even if $q = w^d$,
which differs with the definition we used for the nondegenerate case.
See Section 7 for a claim regarding the degree growth of $f$.

Let the weight of $z^n w^m$ be $n + \alpha m$.
Then the weight of $q$ is $\gamma + \alpha d$ and 
the weight of $Q_z^n(w)$ is $n \gamma d^{n - 1} + \alpha d^n$.
Let $h$ be the weighted homogeneous part of $q$
of highest weight $\gamma + \alpha d$,
which contains $z^{\delta} w^d$.
If $\gamma = 0$, then
$z^{\alpha {\delta}^n} h^n(z^{- \alpha} w)$ is the weighted homogeneous part of 
$Q_z^n (w)$ of weight $\alpha d^n$.
If $\gamma \neq 0$, then
$z^{n \gamma d^{n - 1}} (z^{- \gamma} h(z,w))^{d^{n-1}}$ is the weighted homogeneous part of 
$Q_z^n (w)$ of weight $n \gamma d^{n - 1} + \alpha d^n$.

The dynamics of $\tilde{f}$ on $\mathbb{P} (r,s,1)$ is 
the same as in the case $\delta < d$
if $\gamma \neq 0$ for any positive integers $r$ and $s$.
Because it contracts $L_{\infty} - I_{\tilde{f}}$ to the indeterminacy point $p_{\infty}$,
the point $p_{\infty}$ attracts most nearby points in $A_p \times \mathbb{C}$.

In this subsection we give only one example,
i.e. monomial maps.
If $\gamma = 0$,
then there are many polynomial skew products
that are semiconjugate to polynomial products;
such maps are studied in \cite[Theorem 3.7 and Proposition 3.9]{u-sym}, 
\cite[Examples 5.2 and 5.3]{u-weight} and \cite{u-semiconj}.
However,
we have no such maps if $\gamma \neq 0$.

\begin{ex}[monomial maps]
Let $f = (z^{d}, z^{\gamma} w^d)$ and $\gamma \neq 0$.
Then $\alpha = - \infty$
and $f^n = (z^{d^n}, z^{\gamma_n} w^{d^n})$,
where $\gamma_n = n \gamma d^{n-1} + d^n$.
Hence
\[
G_f = G_z =
\begin{cases}
\infty & \text{ on } \{ |z| > 1, w \neq 0 \} \\
\log^{+} |w| & \text{ on } \{ |z| = 1 \} \\
0 & \text{ on } \{ |z| < 1 \} \cup \{ w = 0 \}.
\end{cases}
\]
Moreover,
\[
\lim_{n \to \infty} \frac{1}{\deg (f^n)} \log^{+} |f^n(z,w)| 
= \lim_{n \to \infty} \frac{1}{n \gamma d^{n-1} + d^n} \log^{+} |Q_z^n(w)| 
\]
\[
=
\begin{cases}
\log^{+} |z| & \text{ on } \{ w \neq 0 \} \\
0 & \text{ on } \{ w = 0 \},
\end{cases}
\]
which is plurisubharmonic on $\mathbb{C}^2$
but not continuous on $\{ |z| \geq 1, w = 0 \}$,
and $G(z,w) = \log |w|$ on $\{ z \neq 0 \}$,
which is continuous and pluriharmonic on $\{ zw \neq 0 \}$
and plurisubharmonic on $\{ z \neq 0 \}$.
\end{ex}

\subsection{Existence of Green functions}
We defined $W_R$ 
as $\{ |z| > R, |w| > R|z|^{\alpha} \}$ if $\gamma \neq 0$. 
If $\alpha = - \infty$, 
then $W_R = \{ |z| > R, |w| > 0 \}$ 
since we may assume that $R > 1$.
As same as the case $\delta < d$, we have the following lemma.

\begin{lem}\label{delta = d: main lem}
If $\delta = d$, 
then $q(z,w) \sim z^{\gamma} w^d$ on $W_R$ for large $R > 0$,
and $f$ preserves $W_R$; that is, $f(W_R) \subset W_R$.
\end{lem}

\begin{proof}
If $q(z,w) \neq b(z)w^d$, then $\alpha > - \infty$
and the proof is the same as the case $\delta < d$,
the proof of Lemma {\rmfamily \ref{delta < d: main lem}}.
If $q(z,w) = b(z)w^d$, 
then $\alpha = - \infty$ and this claim is trivial.
Moreover, 
$q(z,w) \sim z^{\gamma} w^d$ on $\{ |z| > R \} \times \mathbb{C}$.
\end{proof}

From now on, we deal with only the case $\gamma \neq 0$.
The lemma above implies the following two theorems.

\begin{theorem}\label{delta = d: main thm1}
Let $\delta = d$ and $\gamma \neq 0$.
If $\alpha > 0$, then
$G_z = \infty$ on $A_f$
and $\tilde{G}_z \leq \alpha G_p$ on $B_f$.
\end{theorem}

\begin{proof}
By Lemma {\rmfamily \ref{delta = d: main lem}},
there exists a positive constant $r < 1$ such that
$| q(z,w) | > r |z^{\gamma} w^d|$ on $W_R$.
Since $p(z) \sim z^{d}$,
there exists a positive constant $r_0 < 1$ such that
$|p(z)| > r_0 |z|^d$
and so $|p^n(z)| > |r_0 z|^{d^n}$ if $|z| > R$.
Using these inequalities inductively, we get 
\[
|Q_z^n(w)| > r^{1 + d + \cdots + d^{n-1}} |(r_0 z)^{n \gamma d^{n-1}} w^{d^n}|
\]
and so
\[
\dfrac{1}{d^n} \log |Q_z^n(w)| 
> \log r + \dfrac{n \gamma}{d} \log |r_0 z| + \log |w|, 
\]
which tends to $\infty$ as $n \to \infty$.
Therefore, $G_z = \infty$ on $W_R$, which extends to $A_f$.

Let $(z,w)$ be a point in $B_f$.
Then $f^n(z,w)$ never belong to $W_R$
and so $|Q_z^n(w)| < R |p^n(z)|^{\alpha}$.
Hence $\tilde{G}_z (w) \leq \alpha G_p(z)$ on $B_f$.
\end{proof}

The existence of $G_z$ on $B_f$ is still unclear.
We exhibit three examples that relate to this problem.

\begin{ex}
For any positive integer $s$,
let $f = f_s = (z^2, z(w^2 - z^s) + z^{2s})$,
which is conjugate to $f_1 = (z^2, zw^2)$ by $\pi = (z, w + z^s)$.
Then $\alpha = s$ and 
$G_z = \alpha \log |z|$ on $B_f = \{ |z| > 1, w = z^s \}$.
\end{ex}

\begin{ex}
Let $f = (z^r, z^{\gamma} (w^r - z^s) + z^s)$.
Then $\alpha = s/r$ and $G_z = \alpha \log |z|$ 
on $\{ |z| > 1, w^r = z^s \} \subset B_f$.
\end{ex}

\begin{ex}
Let $f = (z^2, zw^2 + z^2w)$.
Then $\alpha = 1$ and 
$G_z = 0$ on $\{ w = 0 \} \subset B_f$. 
\end{ex}

Let us return to the statement on the Green function
in the remaining case $\alpha \leq 0$.

\begin{theorem}
Let $\delta = d$ and $\gamma \neq 0$. 
If $\alpha \leq 0$, then 
$G_z$ is $\infty$ on $A_f$ and $0$ on $B_f$.
\end{theorem}

\begin{proof}
If $q(z,w) \neq b(z)w^d$, then $\alpha > - \infty$.
The proof of the claim $G_z = \infty$ on $A_f$ is the same as 
the proof of Theorem {\rmfamily \ref{delta = d: main thm1}}.
The claim $G_z = 0$ on $B_f$ follows from
the definition of $B_f$ and the assumption $\alpha \leq 0$.
If $q(z,w) = b(z)w^d$, then $\alpha = - \infty$ 
and this claim follows from the direct calculation;
see the proof of 
Proposition {\rmfamily \ref{delta = d: q = b(z) w^d}} below 
for detail.
\end{proof}

More precisely, 
$B_f$ consists of infinitely many lines if $q = b(z)w^d$.

\begin{pro}\label{delta = d: q = b(z) w^d}
Let $\delta = d$ and $\gamma \neq 0$. 
If $q(z,w) = b(z) w^d$, then
$G_z$ is $\infty$ on $A_f$ and $0$ on $B_f$.
Moreover, 
$B_f$ coincides with the union of 
the preimages of $\{ z \in A_p, w = 0 \}$ under $f$,
which is equal to
\[
\bigcup_{n \geq 0} p^{-n}(b^{-1}(0) \cap A_p) \times \mathbb{C} 
\cup \{ z \in A_p, w = 0 \}.
\]
\end{pro}

\begin{proof}
Let $f(z,w) = (p(z), b(z)w^d)$.
Then $f^n(z,w) = (p^n(z), B_n (z) w^{d^n})$,
where $B_n (z) = b(p^{n-1}(z)) \cdots b(p(z))^{d^{n-2}} b(z)^{d^{n-1}}$.
Hence 
\[
\dfrac{1}{d^n} \log |Q_z^n(w)| 
= \log |w| + \sum_{j = 0}^{n - 1} \dfrac{1}{d^{j+1}} \log |b(p^j(z))|.
\]
This finite sum tends to $\infty$ as $n \to \infty$
unless $b(p^j(z)) = 0$ for some $j \geq 0$,
since $\deg p = d$.
\end{proof}

Combining two theorems above, we get the following corollary.

\begin{cor}
If $\delta = d$ and $\gamma \neq 0$, then
\[
G_f^{\alpha} (z,w) =
\begin{cases}
\infty & \text{ on } A_f \\
\max \{ \alpha, 0 \} G_p(z) & \text{ on } B_f.
\end{cases}
\]
\end{cor}

The dynamics of $\tilde{f}$ is similar to 
that mentioned in the previous section. 
The set $A_{\tilde{f}}$ is included in the attracting basin of $p_{\infty}$, 
and $B_{\tilde{f}}$ is the attracting basin of $I_{\tilde{f}} - \{ p_{\infty} \}$.
Hence the same claim as
Proposition {\rmfamily \ref{delta > d: Fatou and Julia}} holds
except the description of $J_{\tilde{f}}$ in terms of a Green function,
for which we use 
the Green function in Corollary {\rmfamily \ref{delta = d; cor2}} below 
instead of $G_z^{\alpha}$. 

Now we show the existence of
other Green functions
that are locally bounded on $\mathbb{C}^{2}$. 

\begin{theorem}
If $\delta = d$ and $\gamma \neq 0$, then
\[
\lim_{n \to \infty} \frac{1}{n \gamma d^{n-1}} \log |Q_z^n(w)| = G_p(z)
\text{ on } A_f.
\]
\end{theorem}

\begin{proof}
By Lemma {\rmfamily \ref{delta = d: main lem}},
there exist constants $r_1 < 1 < r_2$ such that 
$r_1 |z^{\gamma} w^d| < |q(z,w)| < r_2 |z^{\gamma} w^d|$ on $W_R$.
By using $|q(z,w)| < r_2 |z^{\gamma} w^d|$ inductively,
we get the upper estimate
\[
|Q_z^n(w)| < r_2^{1 + d + \cdots + d^{n-1}} |p^{n-1} (z)|^{\gamma} 
|p^{n-2} (z)|^{\gamma d} \cdots |z|^{\gamma d^{n-1}} |w|^{d^n}
\]
and so
\[
\frac{1}{n \gamma d^{n-1}} \log |Q_z^n(w)| 
< \dfrac{d}{n \gamma} \left( \log r_2 + \log |w| \right) 
\]
\[
+ \dfrac{1}{n} \left\{ \log |z| + \dfrac{1}{d} \log |p(z)| 
+ \cdots + \dfrac{1}{d^{n-1}} \log |p^{n-1}(z)| \right\}.
\]
Since $d^{-n} \log|p^{n}|$ converges to $G_p$ on $A_p$,
the right hand side converges to $G_p$ as $n \to \infty$.
Thus the inequality
\[
\limsup_{n \to \infty} \frac{1}{n \gamma d^{n-1}} \log |Q_z^n(w)| \leq G_p(z)
\]
holds on $W_R$.
By the same argument in terms of $r_1 |z^{\gamma} w^d| < |q(z,w)|$,
we get the inverse inequality.
Therefore, we get the required equation on $W_R$.

A similar argument as above induces the required equation on $A_f$,
because if $f^N (z,w)$ belongs to $W_R$
then $Q_z^n(w) = Q_{p^N(z)}^{n - N} (Q_z^N(w))$ approximates to 
$p^{n-1} (z)^{\gamma} p^{n-2} (z)^{\gamma d} \cdots p^N (z)^{\gamma d^{n-N-1}} Q_z^N(w)^{d^{n-N}}$
for $n \geq N$.
\end{proof}

It follows from this theorem that 
$\lim_{n \to \infty} (n \gamma d^{n-1})^{-1} \log^{+} |Q_z^n| = 0$ on $B_f$.
Therefore,

\begin{cor}\label{delta = d; cor2}
If $\delta = d$ and $\gamma \neq 0$, then
\[
\lim_{n \to \infty} \frac{1}{n \gamma d^{n-1}} \log^{+} |f^n(z,w)| = 
\lim_{n \to \infty} \frac{1}{n \gamma d^{n-1}} \log^{+} |Q_z^n(w)|  
\]
\[
= 
\begin{cases}
G_p(z) & \text{ on } \mathbb{C}^{2} - B_f \\
0      & \text{ on } B_f.
\end{cases}
\]
\end{cor}

Finally, we prove the existence of $G$ using
Lemma {\rmfamily \ref{delta = d: main lem}}.
By definition,
$G(f^n(z,w)) = d^n G(z,w) + n \gamma d^{n-1} G_p(z)$
if it exists.

\begin{theorem}
The limit $G$ is defined, continuous and pluriharmonic on $A_f$.
Moreover, 
$G = \log|w| + o(1)$ on $W_R$,
$G \sim \log |w|$ as $w \to \infty$ for fixed $z$ in $A_p$,
and $G$ tends to $- \infty$ as $(z,w)$ in $A_f$ tends to 
any point in $\partial A_f - J_p \times \mathbb{C}$. 
\end{theorem}

\begin{proof}
The proof is similar to those of Theorems
{\rmfamily \ref{delta > d: main thm}} and
{\rmfamily \ref{delta < d: main thm}}.
We first show the uniform convergence to $G$ on $W_R$.
Let $G_n = d^{-n} \log |z_n^{- n \gamma /d} w_n|$,
where $(z_n,w_n) = f^n(z,w)$.
By Lemma {\rmfamily \ref{delta = d: main lem}},
there are constants $r > 1$ and $r_0 > 1$ such that
$|w_{n+1}| < r|z_n^{\gamma} w_n^d|$ and $|z_n^d| < r_0 |z_{n+1}|$ on $W_R$.
Hence 
\[
G_{n+1} - G_n = 
\dfrac{1}{d^{n+1}} \log \left| \dfrac{w_{n+1}}{w_n^d} \cdot 
\dfrac{z_n^{n \gamma}}{z_{n+1}^{(n+1) \frac{\gamma}{d}}} \right|
< \dfrac{1}{d^{n+1}} \log \left| r \left( \dfrac{z_n^d}{z_{n+1}} \right)^{(n+1) \frac{\gamma}{d}} \right|
\]
\[
< \dfrac{1}{d^{n+1}} \log \left| r {r_0}^{(n+1) \frac{\gamma}{d}} \right|
= \dfrac{1}{d^{n+1}} \log r + \dfrac{n+1}{d^{n+1}} \cdot \dfrac{\gamma}{d} \log r_0.
\]
Therefore, $G_n$ converges uniformly to $G$ on $W_R$,
which is continuous and pluriharmonic.
By the inequality above,
\begin{equation}\label{delta = d: eq}
|G(z,w) - \log |w|| < C_R \text{ on } W_R,
\end{equation}
where the constant $C_R > 0$ converges to $0$ as $R \to \infty$.
The equation $G = d^{-n} G \circ f^n - n \gamma d^{-1} G_p$
extends the domain of $G$ from $W_R$ to $A_f$. 

Next, we show the last statement.
By inequality (\ref{delta = d: eq}),
\[
\left |G(f^n(z,w)) - \log R|p^n(z)|^{\alpha} \right| < C_R \text{ on } f^{-n}(E),
\]
where $E = \{ |w| = R|z|^{\alpha}, |z| > R \} \subset \partial W_R$.
Thus, by the identity $G \circ f^n = d^n G + n \gamma d^{n-1} G_p$,
\[
\left |G(z,w) + n \dfrac{\gamma}{d} G_p(z) - \dfrac{\log R|p^n(z)|^{\alpha}}{d^n} \right| 
< \dfrac{C_R}{d^n} \text{ on } f^{-n}(E).
\]
Therefore, the values of $G$ on $f^{-n}(E)$ tend to $- \infty$ as $n \to \infty$. 
\end{proof}

This theorem guarantees that
$B_f \cap ( \{ z \} \times \mathbb{C} ) \neq \emptyset$
for any $z$ in $A_p$. 
Since $G = - \infty$ on $B_f$, 
it follows that $G$ is defined and plurisubharmonic on $A_p \times \mathbb{C}$, 
and $A_f = \{ (z,w) \in A_p \times \mathbb{C} : G > - \infty \}$. 
The identity $G \circ f^n = d^n G + n \gamma d^{n-1} G_p$ induces that 
$G(z,w)$ converges to $0$ as $(z,w)$ in $A_f$ tends to the intersection of 
the closure of $\partial A_f - J_p \times \mathbb{C}$ 
and $J_p \times \mathbb{C}$.
Hence the description of the last statement is different 
from those in the case $\delta \neq d$.

\section{Degree growth} 

In this section we give 
a remark on the relationship between algebraically stability and weight growth,
the list of inequalities of degree growth in terms of $\alpha$,
and a corollary on the existence of a weighted Green function 
that is normalized by $\deg (f^n)$.

\subsection{Algebraically stability and Weight growth}

Assume that $f$ is not a polynomial product.
We saw that if $\gamma \neq 0$ or $\delta > d$,
then $\alpha$ is a positive rational number,
$f$ extends to an algebraically stable map on $\mathbb{P} (r,s,1)$,
where $s/r = \alpha$,
and the weight of $Q_z^n$ is $\alpha \lambda^n$.   
Hence the weight of $f^n$ is equal to $\lambda^n$
if we define it as the maximum of 
the degree of $p^n$ and the weight of $Q_z^n$ times $\alpha^{-1}$;
namely,
$\text{weight} (f^n) = ( \text{weight} f )^n$.
This is an analogue of the well known fact:
if $\tilde{f}$ is algebraically stable on $\mathbb{P}^2$,
then $\deg (f^n) = (\deg f)^n$.

\subsection{Degree growth}
Next, we give the list of inequalities on the degree growth of $f$,
which follows from the definition of $\alpha$. 

\begin{enumerate} 
\item If $\delta > d$, then 
$\delta^n \leq \deg (f^n) \leq \max \{ \alpha, 1 \} \delta^n$.
\item If $\delta < d$ and $\alpha > 0$, then
\[
\left[ 1 + \dfrac{\gamma}{d - \delta} 
\left\{ 1 - \left( \dfrac{\delta}{d} \right)^n \right\} \right] d^n
\leq \deg (f^n)
\]
\[
\leq \left[ \max \{ \alpha, 1 \} + \max \left\{ \dfrac{1}{\alpha}, 1 \right\} 
\dfrac{\gamma}{d - \delta} \left\{ 1 - \left( \dfrac{\delta}{d} \right)^n \right\} \right] d^n.
\]
If $\delta < d$ and $\alpha \leq 0$, then 
\[
\deg (f^n) = \left[ 1 + \dfrac{\gamma}{d - \delta} 
\left\{ 1 - \left( \dfrac{\delta}{d} \right)^n \right\} \right] d^n.
\]
\item If $\delta = d$ and $\alpha > 0$, then 
\[
\left( \dfrac{\gamma}{d} n + 1 \right) d^n
\leq \deg (f^n) \leq 
\left[ \max \left\{ \dfrac{1}{\alpha}, 1 \right\} \dfrac{\gamma}{d} n + \max \{ \alpha, 1 \} \right] d^n.
\]
If $\delta = d$ and $\alpha \leq 0$, then
\[
\deg (f^n) = \left( \dfrac{\gamma}{d} n + 1 \right) d^n.
\]
\end{enumerate}

Note that it follows from the equalities above that,
if $\gamma = 0$, then 
$\lambda^n \leq \deg (f^n) \leq \max \{ \alpha, 1 \} \lambda^n$,
which was already stated in \cite{u-weight}.
In general,  
it is proved in \cite{fj-ev} that,
for any polynomial map $F$ 
which is not conjugate to a skew product 
such that $\delta = d$ and $\gamma \neq 0$,
if the dynamical degree $\lambda$ is larger than $1$, 
then there is $D \geq 1$ 
such that $\lambda^n \leq \deg (F^n) \leq D \lambda^n$.

\subsection{Another Green function}
In the final subsection
we consider the existence of the limit
\[
\lim_{n \to \infty} \frac{1}{\deg (f^n)} \log^{+} |f^n(z,w)|_{\alpha}.  
\]
If $\delta \neq d$, 
then the topological degree $\delta d$ is smaller 
than $\lambda^2$. 
In this case, the existence of the limit follows from 
Corollaries {\rmfamily \ref{delta > d; main cor}} 
and {\rmfamily \ref{delta < d; main cor}}
and the main result in \cite{bfj},
which can be restated as follows:
for a polynomial map $F$ 
whose topological degree is smaller 
than the square $\lambda^2$ of the dynamical degree,
the ratio of $\deg (F^n)$ and $\lambda^n$ 
converges to some positive number.
For the case $\delta = d$,
the existence follows if $\alpha \leq 0$
because of Corollary {\rmfamily \ref{delta = d; cor2}}
and the result above on the degree growth.
Consequently,

\begin{cor}
If $\delta \neq d$ or $\alpha \leq 0$, then
the limit of $(\deg (f^n))^{-1} \log^{+}$ $|f^n(z,w)|_{\alpha}$
is defined on $\mathbb{C}^2$.
Moreover, it is continuous on $\mathbb{C}^2$ if $\delta > d$,
on $A_p \times \mathbb{C}$ if $\delta < d$,
and on $\mathbb{C}^2 - B_f$ if $\delta = d$ and $\alpha \leq 0$.
\end{cor}


\end{document}